\documentclass[12pt]{article}
\usepackage{geometry}
\usepackage{youngtab}
\usepackage{tikz}
\usepackage{placeins}

\usetikzlibrary{arrows,shapes,positioning}
\usetikzlibrary{calc,decorations.markings}
\tikzstyle{block} = [rectangle, draw, rounded corners, text centered, minimum height = 2em, minimum width = 7em, align=left, scale = 0.75]

\tikzstyle{line} = [draw, thick]

\usetikzlibrary{calc,decorations.markings}
\usepackage{hyperref}
\usepackage{amsthm}
\usepackage{makecell}
\usepackage{enumitem}
\usepackage{graphicx}

\usepackage{subfig}

\usepackage{amssymb}
\usepackage{epstopdf}
\usepackage{amsmath}
\usepackage{amsfonts}

\usepackage[longnamesfirst]{natbib}
\bibpunct[ ]{(}{)}{,}{a}{}{,}

\theoremstyle{plain}
\newtheorem{theorem}{Theorem}[section]
\newtheorem{corollary}[theorem]{Corollary}

\newtheorem{prop}[theorem]{Proposition}
\newtheorem{lemma}[theorem]{Lemma}
\newtheorem{remark}[theorem]{Remark}

\theoremstyle{definition}
\newtheorem{definition}[theorem]{Definition}
\newtheorem{example}[theorem]{Example} 

\theoremstyle{plain}

\newcommand{\N}{\mathbb{N}}

\newcommand{\p}{\mathbb{P}}
\newcommand{\E}{\mathbb{E}}
\newcommand{\expt}{\mathbb{E}}
\newcommand{\indic}{\mathbf{1}}

\newcommand{\floor}[1]{{\left\lfloor #1 \right\rfloor}}

\newcommand{\sset}{\subset}

\newcommand{\al}{\alpha}

\newcommand{\Om}{\Omega}
\newcommand{\mathforall}{\text{ for all }}

\newcommand{\mathand}{\;\text{and}\;}

\newcommand{\mathas}{\;\text{as}\;}

\newcommand{\ga}{\gamma}
\newcommand{\Ga}{\Gamma}
\newcommand{\ep}{\epsilon}

\newcommand{\de}{\delta}

\newcommand{\sig}{\sigma}
\DeclareMathOperator*{\Arg}{Arg}

\newcommand{\del}{\partial}

\newcommand{\scrA}{\mathcal{L}}

\newcommand{\scrE}{\mathcal{E}}
\newcommand{\scrC}{\mathcal{C}}

\newcommand{\scrL}{\mathcal{L}}

\newcommand{\scrF}{\mathcal{F}}

\newcommand{\close}[1]{\mkern 1.5mu\overline{\mkern-1.5mu#1\mkern-1.5mu}\mkern 1.5mu}

\newcommand{\Z}{\mathbb{Z}}
\newcommand{\C}{\mathbb{C}}

\newcommand{\R}{\mathbb{R}}

\newcommand{\eqd}{\stackrel{d}{=}}

\newcommand{\cvgd}{\stackrel{d}{\to}}

\newcommand{\X}{\times}

\newcommand{\smin}{\setminus}
\newcommand{\lf}{\left}
\newcommand{\rg}{\right}

\newcommand{\bs}{\mathbf{s}}

\newcommand{\be}{\beta}

\author{Duncan Dauvergne \and Mihai Nica \and B\'alint Vir\'ag}

\title{Uniform convergence to the Airy line ensemble}

\begin{document}
\maketitle
\begin{figure}[h]
\vspace{-3.5em}
\begin{center}
\includegraphics[width=0.8\textwidth]{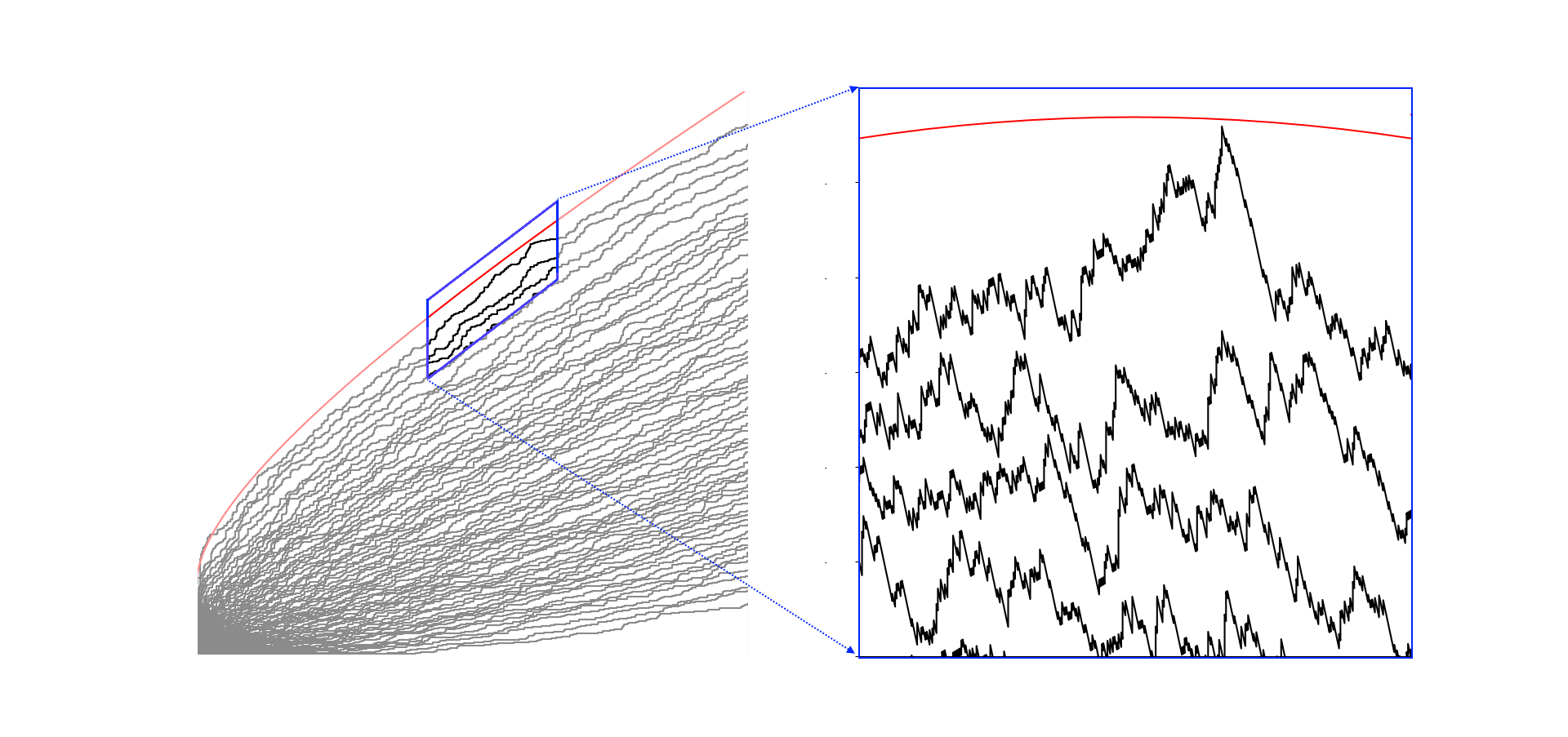}
\end{center}
\vspace{-2em}
\end{figure}

\begin{abstract} We show that classical integrable models of last passage percolation and the related nonintersecting random walks converge uniformly on compact sets to the Airy line ensemble. Our core approach is to show convergence of nonintersecting Bernoulli random walks in all feasible directions in the parameter space. We then use coupling arguments to extend convergence to other models.
\end{abstract}

\section{Introduction}
The Airy line ensemble is a random sequence of continuous functions that arises as a scaling limit in random matrix theory and  other models within the KPZ universality class. In the last passage percolation setting it was constructed by \cite{prahofer2002scale} as a scaling limit of the polynuclear growth model, see also \cite{macedo1994universal} and \cite{forrester1999correlations}. \cite{prahofer2002scale} showed that the finite dimensional distributions of an appropriately centered and rescaled version of the multi-layer polynuclear growth model converge to those of the Airy line ensemble.

\cite{CH} showed that appropriate statistics in Brownian last passage percolation converge to the Airy line ensemble in the topology of uniform convergence of functions on compact sets. This stronger notion of convergence allowed them to prove new and interesting qualitative properties of the Airy line ensemble.

Recently, \cite{DOV} constructed the Airy sheet, the two-parameter scaling limit of Brownian last passage percolation, in terms of the Airy line ensemble. The Airy sheet was used to build the full scaling limit of Brownian last passage percolation, the directed landscape. For these results, uniform convergence to the Airy line ensemble (rather than just convergence of finite dimensional distributions) is a crucial input. In fact, this convergence is the \emph{only} input necessary for an i.i.d.\ last passage model to also converge to both the Airy sheet and the directed landscape. We prove this in the forthcoming work \cite{DV2}.

With this motivation in mind, we devote this paper to proving uniform convergence to the Airy line ensemble for various classical models. In this setting, there is a large literature on convergence of finite-dimensional distributions. The contribution of this paper is a unified approach which applies in all feasible directions of the parameter space and a general argument giving uniform convergence for these models.

\subsection*{Main results and an overview of the proofs}

Consider an infinite array $W =  (W_{i, j})_{i, j \in \N}$ of nonnegative real numbers. For a point $(m, n) \in \N \X \N$, the last passage value $L_n(m)$ in the array $W$ is the maximum weight of an up-right path (the sum of the entries along that path) from the corner $(1, 1)$ to the point $(m, n)$. Last passage percolation can also be done with several disjoint paths. The $k$-path last passage value $L_{n,k}(m)$ is the maximum sum of weights of $k$ disjoint up-right paths with start and end points $(1, i)$ and $(m, n - k + i)$ for $i = 1, \dots, k$. See Figure \ref{fig:LPP_stab} for an illustration and Definition \ref{D:LPP-discrete} for a more precise description.

If we set $L_{n,0} \equiv 0$, the increments $L_{n,k+1}(m) - L_{n,k}(m)$ are nonincreasing in $k$ for any point $(m,n) \in \N^2$. This can be proven directly by manipulating collections of disjoint paths, or alternately is a consequence of Greene's theorem for the RSK correspondence which relates these increments to the row lengths in a Young diagram, see \cite{sagan2013symmetric}. Allowing $m$ to vary, we thus obtain an ordered sequence of functions. When the array $W$ is filled with i.i.d. geometric random variables this sequence has a well-known integrable structure, which makes the model amenable to analysis. Our first theorem is a general convergence result for these functions.

\begin{figure}[!ht]
     \subfloat[$\beta = 1,n=5$]{%
       \includegraphics[width=0.32\textwidth]{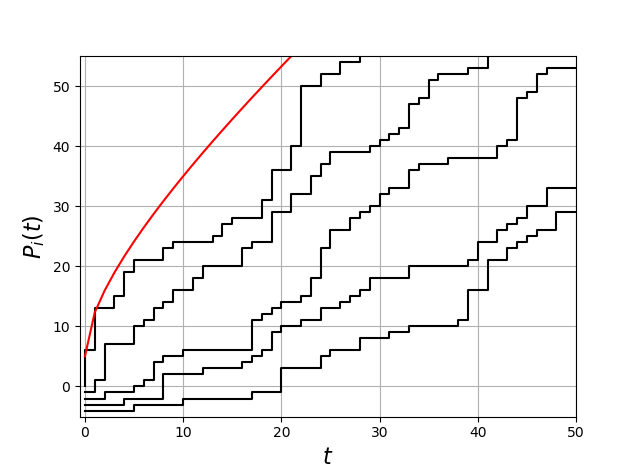}
     }
     \hfill
     \subfloat[$\beta = 2, n=5$]{%
       \includegraphics[width=0.32\textwidth]{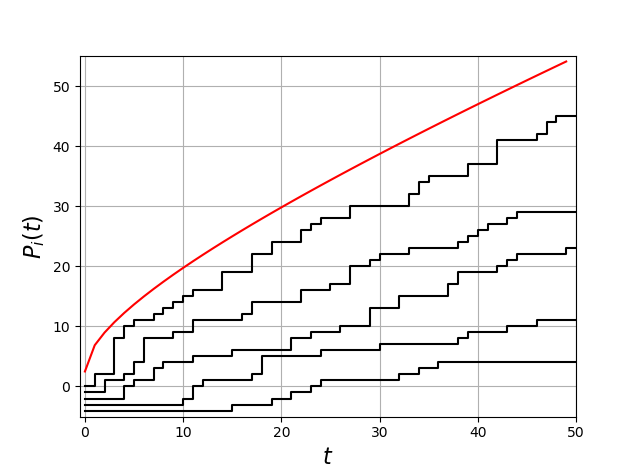}
     }
     \hfill
     \subfloat[$\beta = 1, n=25$]{%
       \includegraphics[width=0.32\textwidth]{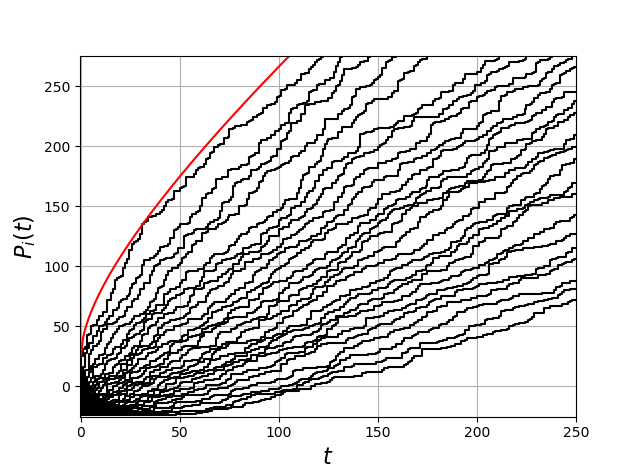}
     }
     \caption{Realizations of differences in last passage percolation in an environment of i.i.d. geometric random variables: $P_k(t) := L_{n,{k+1}}(t) - L_{n,{k}}(t) - k + 1$. These walks are identical in distribution to $n$ random walks whose increments are geometric random variables of mean $\beta^{-1}$ that are conditioned not to intersect for all time, see Section \ref{S:LPP-geometric}. The arctic curve is displayed in red.  Theorem \ref{T:main-lpp} describes the fluctuation limit.}
     \label{fig:NI_geo}
   \end{figure}

\begin{theorem}
\label{T:main-lpp}
Consider a sequence of  last passage percolation models, indexed by $n \in \mathbb{N}$, with independent geometric random variables of mean $\beta_n^{-1} \in (0, \infty)$. Let $m_n$ be a sequence of positive integers: we will analyze last passage values (defined precisely in  \eqref{D:LPP-k}) from the bottom-left corner $(1,1)$ to points near $(m_n,n)$ in these environments. For each $n,m\in \mathbb{N}$ and $\beta \in (0,\infty)$, define the arctic curve:
\begin{equation}\label{E:arctic1}
g_{n,\beta}(m) = (m + n)\beta^{-1} + 2 \sqrt{mn \beta^{-1}(1+\beta^{-1})},
\end{equation}
which is the deterministic approximation of the last passage value $L_{n, 1}(m)$.
We now define the temporal and spatial scaling parameters $\tau_n$ and $\chi_n$ in terms of the value of the arctic curve $g=g_{n,\beta_n}$ and its derivatives $g',g''$ evaluated at $m_n$:
\begin{equation}
\label{E:tau-xi-geom}
\tau_n^3 = \frac{2 g '(1 + g')}{g''^2}, \qquad \chi_n^3 = \frac{ [g'(1+g')]^{2}}{ -2g''}.
\end{equation}
Also, let $h_n$ be the linear approximation of the arctic curve $g$ at $m_n$:
$$
h_n(m) = g + (m - m_n)g'
$$
Then the following statements are equivalent:
\begin{enumerate}[label=(\roman*)]
	\item The dimensions of the last passage grid and the mean total sum of the weights in the grid converge to $\infty$:
	\begin{equation*}
	n\to\infty, \qquad m_n\to\infty, \qquad \frac{nm_n}{\beta_n}\to\infty.
	\end{equation*}
	\item The rescaled differences of the $k$-path and $(k-1)$-path last passage values converge in distribution\footnote{The figure above the abstract of the paper illustrates the rescaling and convergence graphically}, uniformly over compact sets of $\mathbb{N}\times \mathbb{R}$, to the parabolic Airy line ensemble $\scrA$ (see Definition \ref{D:airy-line} for a precise definition of the parabolic Airy line ensemble).
	$$\frac{(L_{n,k}-L_{n,k-1}-h_n)(m_n+\lfloor \tau_{n}t\rfloor)}{\chi_n}\Rightarrow \scrA_k(t).
	$$
\end{enumerate}

\end{theorem}

In the statement of Theorem \ref{T:main-lpp}, we have specified that we prove convergence to the parabolic Airy line ensemble $\scrA$, which is related to the usual Airy line ensemble $\mathcal A$ by the addition of a parabola: $\mathcal A(t) = \mathcal L(t) + t^2$. The process $\mathcal A$ is stationary, and hence $\mathcal L$ has a parabolic shape. There are various reasons for focusing on $\mathcal L$ as opposed to $\mathcal A$, e.g. $\mathcal L$ satisfies the Brownian Gibbs property and is the object directly used in the recent construction of the Airy sheet.

See Remark \ref{r:11proof} for the proof of Theorem \ref{T:main-lpp}.

\begin{remark}
Equation \eqref{E:arctic1} for the arctic curve has the form $$
g = \mu + 2 \sigma,
$$
where $\mu = (m+n) \beta^{-1}$ is the expected weight of any individual up-right path, and $\sigma = \sqrt{mn \beta^{-1}(1+\beta^{-1})}$ is standard deviation of the sum of \emph{all} the random variables reachable by any path. The same form for the arctic curve also holds for all the limiting environments we consider in Section \ref{S:Corollaries}. From this formula, one can easily see that the shape of $g$ depends only on the aspect ratio of the rectangle $m/n$ in the sense that:
$$
g_{n,\beta}(m) = n \, g_{1,\beta}\left(\frac{m}{n}\right).
$$
\end{remark}

\begin{remark}
Another equivalent condition to  $(i)$ and $(ii)$ is that  some  scaled  distributional limit of  $L_{n,1}(m_n)$ is the Tracy-Widom law. This also follows from our proof.
\end{remark}

The proof of Theorem \ref{T:main-lpp} goes by relating last passage percolation to nonintersecting walks. For each $n$, by a theorem of \cite{o2003conditioned}
$$
L_{n, k} - L_{n, k-1} - k + 1, \qquad  i=1,\ldots, n
$$
is equal in law to $n$ nonintersecting geometric random walk paths, see Section \ref{S:LPP-geometric} for a precise equivalence. Considerations regarding these random walks gives rise to the scaling parameters $\tau_n$ and $\chi_n$, which may appear somewhat complicated and mysterious at first glance. In fact, they are derived as the unique positive solutions of the following system of equations:
\begin{align}
\label{E:gb-1}
g'(1 + g') \tau_n &= 2 \chi_n^2, \\
\label{E:gb-2}
\frac{-g''}{2} \frac{\tau_n^2}{\chi_n} &= 1.
\end{align}

This system of equations comes from probabilistic considerations about nonintersecting random walks, see the exposition in Section \ref{S:Bernoulli} for the full derivation and details. Intuitively, the top line behaves like a geometric random walk of mean $g'$, and the term $g'(1+g')$ appears as the variance of a single step. The $g''$ term comes matching the curvature of the arctic curve $g$ to the desired parabolic shape of the parabolic Airy line ensemble. The temporal and spatial scaling parameters are completely determined by matching both the Brownian variance and the limiting curvature.

One of the strengths of allowing both the parameters $m_n$ and $\be_n$ to vary arbitrarily in Theorem \ref{T:main-lpp} and of showing uniform convergence rather than just finite dimensional distribution convergence is that we can easily handle convergence of other integrable models of last passage percolation by coupling.

\begin{corollary}
\label{C:intro-cor-lpp} The convergence in Theorem \ref{T:main-lpp} also holds for exponential and Brownian last passage percolation, as well as for Poisson last passage percolation both on lines and in the plane.
\end{corollary}

See Section \ref{S:Corollaries} for precise definitions, statements, and scaling relations for the above corollary. As in Theorem \ref{T:main-lpp}, we prove convergence in all feasible parameter directions.

Theorem \ref{T:main-lpp} relies on a convergence theorem for nonintersecting Bernoulli walks. See Figure \ref{fig:NI_bernoulli} and Section \ref{S:Bernoulli} for the precise definition.

\begin{figure}[!ht]
     \subfloat[$\beta = 1,n=5$]{%
       \includegraphics[width=0.32\textwidth]{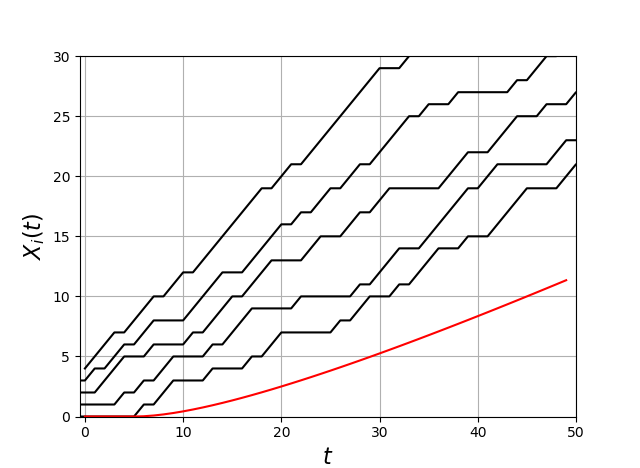}
     }
     \hfill
     \subfloat[$\beta = 2, n=5$]{%
       \includegraphics[width=0.32\textwidth]{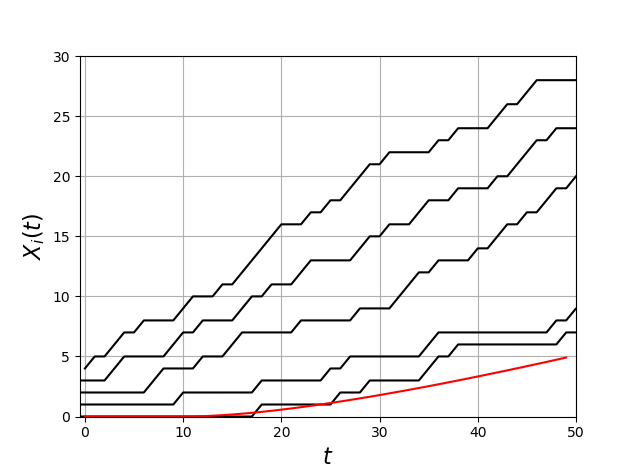}
     }
     \hfill
     \subfloat[$\beta = 1, n=25$]{%
       \includegraphics[width=0.32\textwidth]{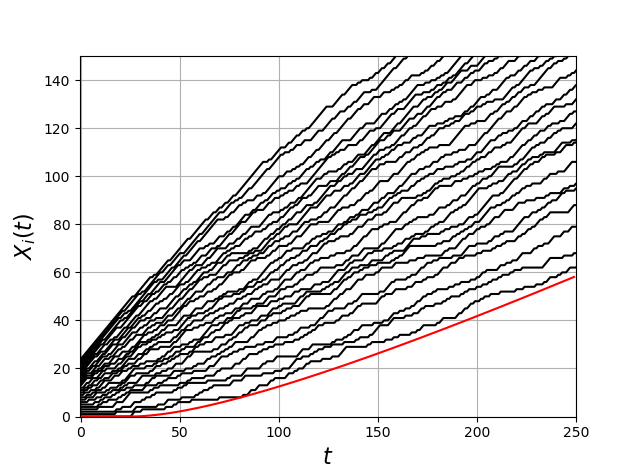}
     }
     \caption{Realizations of nonintersecting Bernoulli random walks for different parameters $\beta$ and $n$.  The arctic curve is shown in red. A contour integral formula allows computation of the fluctuations around the arctic curve in Theorem \ref{T:main-walk}.}
     \label{fig:NI_bernoulli}
   \end{figure}

\begin{theorem}[Nonintersecting Bernoulli walks]
\label{T:main-walk}
Consider sequences of parameters $\beta_n \in (0, \infty)$, $m_n \in \N$ with $m_n \be_n > n$. Let $X_{n,1}(\cdot)<\cdots<X_{n,n}(\cdot)$ be $n$ Bernoulli random walks with mean $\be/(1+ \be)$ started from the initial condition $(0, 1, \dots, n-1)$ and conditioned to never intersect. Define the arctic curve
$$
\gamma_{n, \beta}(m) = \frac{(\sqrt{ m\beta} -\sqrt{n})^2}{1 + \beta} \indic(m \be > n),
$$
the deterministic approximation of the lowest walk $X_{n, 1}(m).$
We define scaling parameters $\chi_n$ and $\tau_n$ in terms of $\ga =\ga_{n, \be_n}$ and its derivative $\ga'$, $\ga''$ evaluated at the point $m_n$:
\begin{equation}
\label{E:tau-xi-bern}
\tau_n^3 = \frac{2\gamma'(1 - \gamma')}{(\ga'')^2}, \qquad \chi_n^3 = \frac{[\gamma'(1 - \gamma')]^2}{2\ga''}.
\end{equation}
Also, let $h_n$ be the linear approximation of $\ga$ at $m_n$.
$$
h_n(m) = \ga + (m - m_n)\ga'
$$
Then the following are equivalent:
\begin{enumerate}[label=(\roman*)]
	\item $\chi_n \to \infty$ with $n$.
	\item The rescaled walks converge in distribution, uniformly over compact sets of $\mathbb{N}\times \mathbb{R}$, to the  parabolic Airy line ensemble $\scrA$ (see Definition \ref{D:airy-line} for a precise definition of the parabolic Airy line ensemble):
	$$
	\frac{(h_n-X_{n,k})(m_n + \lfloor\tau_n t\rfloor)}{\chi_n}\Rightarrow \scrA_k(t).
	$$
\end{enumerate}
\end{theorem}
Theorem \ref{T:main-walk} is proved in Section \ref{ss:Bernoulli}. The arctic curve can be expressed in terms of the probability $p$ of the Bernoulli walks taking an up step and the probability $q = 1 - p$ of taking a flat step. We then get the following expression for the arctic curve:
\begin{equation*}
\ga(m) = (\sqrt{mp} - \sqrt{nq})^2 \indic(mp > nq).
\end{equation*}

After a linear transformation of the graphs, the Bernoulli walks map to geometric walks. Thus Theorem \ref{T:main-walk} can be used to prove Theorem \ref{T:main-lpp}. By equivalence to the classical last passage models discussed above, we also get a version of Theorem \ref{T:main-walk} for other nonintersecting random walk ensembles. The following `meta-corollary' is intentionally vague -- we refer the reader to Section \ref{S:Corollaries} for precise statements. 
\begin{corollary}
	The convergence in Theorem \ref{T:main-walk} also holds for nonintersecting  geometric, exponential, and Poisson walks, as well as for nonintersecting Brownian motions.
\end{corollary}
The nonintersecting Bernoulli random walks appear in the Sepp\"al\"ainen-Johansson last passage model, and our results thus apply in this case, see Corollary \ref{C:seppalainen}.

The starting point for the proof of Theorem \ref{T:main-walk} is a determinantal formula for nonintersecting Bernoulli walks with a kernel given in terms of contour integrals, see \eqref{E:rw-kern} and \eqref{E:rw-kern-2}. Formulas for this process essentially first appeared in \cite{johansson2005arctic} (see also \cite{johansson2001discrete}), but the precise one we apply comes from \cite{borodin2013markov}. We establish convergence of the finite dimensional distributions to those of the Airy line ensemble by taking a limit of this formula. This has been done in many related contexts. Convergence of such formulas is usually handled by a steepest analysis around a double critical point.

The main distinction between our analysis and prior work is that for us, the parameters $m_n/n$ and $\be_n$ can vary with $n$. This causes difficulties that are not there in the fixed parameter case. We deal with this by choosing contours depending on the parameters using careful geometric considerations. We make the connection between the kernel and the probabilistic features of the models apparent by using physical intuition to guide the analysis.

To go from convergence of finite dimensional distributions to uniform convergence requires a tightness argument for nonintersecting random walks. In the context of nonintersecting Brownian motions, tightness was proven in \cite{CH} by exploiting the Brownian Gibbs property (see also \cite{DV} for an alternate proof). Here we give a concise and general proof of tightness that applies to both nonintersecting random walk ensembles and nonintersecting Brownian motions.

\subsection*{Related work}

There is a large body of literature on last passage percolation and nonintersecting random walks in relation to the Airy line ensemble. This is a very partial review of the literature, with results most directly related to the present work. The interested reader should see the review articles \cite{corwin2016kardar, ferrari2010random, quastel2011introduction, takeuchi2018appetizer} and the books \cite{romik2015surprising, weiss2017reflected} for a broader introduction to the area.

\cite{prahofer2002scale} identified the Airy line ensemble as the limit of the multi-layer polynuclear growth model. Their work built on the work of \cite{baik1999distribution} which finds the limit of the length of the longest increasing subsequence of a uniform random permutation, see also \cite{johansson2000shape}.

\cite{prahofer2002scale} proved convergence of finite dimensional distributions. In the context of last passage percolation in the geometric environment along an antidiagonal, \cite{johansson2003discrete} strengthened this to convergence in the uniform-on-compact topology for the top line $\scrA_1$. \cite{CH} proved uniform-on-compact convergence for the whole Airy line ensemble in the context of Brownian last passage percolation.

The Airy line ensemble has also been identified as the limit of many other models, e.g. \cite{ferrari2003step, okounkov2003correlation, johansson2005arctic, borodin2006stochastic, imamura2007dynamics, borodin2008asymptotics, petrov2014asymptotics}. Many of these papers focus only on proving convergence to the Airy process $\scrA_1$. However, the analysis required for proving convergence to the whole Airy line ensemble is essentially the same.

The Gibbs property for ensembles of Brownian motions and random walks has also proven useful for showing tightness of positive temperature analogues of the models in this paper. \cite{corwin2016kpz} used such methods to prove tightness of the sequence of functions coming from the O'Connell-Yor directed polymer model and analyze the limiting KPZ line ensemble.  \cite{corwin2018transversal} used such methods to prove tightness and transversal fluctuation results about asymmetric simple exclusion and the stochastic six vertex model.

The explicit relationship between nonintersecting random walks and last passage percolation that we use is from \cite{o2003conditioned}, which builds on work from \cite{o2002representation} and \cite{konig2002non}. This relationship has various elegant generalizations to related problems, see \cite{biane2005littelmann, o2012directed}.

\subsection*{Organization of the paper}

In Section \ref{S:Bernoulli}, we give a precise definition of nonintersecting Bernoulli walks and derive the scaling parameters in Theorem \ref{T:main-walk} using probabilistic reasoning. In Section \ref{S:FDD}, we perform the asymptotic analysis required to prove convergence of finite dimensional distributions for nonintersecting Bernoulli walks. In Section \ref{S:Uniform}, we present a general tightness argument that allows us to upgrade to uniform convergence in Theorem \ref{T:main-walk}. In Section \ref{S:LPP-geometric}, we formally introduce last passage percolation and translate Theorem \ref{T:main-walk} to get Theorem \ref{T:main-lpp}. In Section \ref{S:Corollaries}, we prove corollaries related to other models by using appropriate couplings.

\section{Nonintersecting Bernoulli walks and the Airy line ensemble}\label{S:Bernoulli}

For $\be \in (0, \infty)$, a random function $X:\N \to \Z$ is a Bernoulli random walk if it has independent increments $X(m+1) - X(m)$ with Bernoulli distribution with mean $\be/(1+\be)$. The parameter $\be$ itself is the ratio of up steps to flat steps, and will be called the {\bf odds}.  This particular parameter makes the analysis of contour integrals cleaner.

 A collection $X_1(\cdot), X_2(\cdot), \dots, X_n(\cdot)$ are \textbf{nonintersecting Bernoulli walks} if each of the $X_i$s are independent Bernoulli walks with odds $\beta$ started from the initial condition $X_i(0) = i-1$, conditioned so that
$$
X_1(m) < X_2(m) <  \dots < X_n(m) \qquad \text{
for all }m \in \N.
$$
Since this is a measure $0$ event, this must formally be defined so the above equation holds for all $m \le m_0$, and then $m_0$ is taken to $\infty$. This setup is also known as the Krawtchouk ensemble and the walks can alternatively be described in terms of a Doob transform involving a Vandermonde determinant, see \cite{konig2002non} for discussion.

Theorem \ref{T:main-walk} says that the scaling limit of the edge of nonintersecting Bernoulli walks is the parabolic Airy line ensemble. 

\begin{definition}
	\label{D:airy-line}
	The {\bf parabolic Airy line ensemble} $\scrA = (\scrA_1, \scrA_2, \dots)$ is a sequence of random continuous functions uniquely characterized by the following two properties:
	\\1. (Non-intersecting) For every $t \in \R$, we have (almost surely) that $\scrA_1(t) > \scrA_2(t) > \dots$.
	\\2. (Determinantal) For any finite set of times $t_1 < t_2 < \dots < t_k$, the set of points
	$$
	\{\scrA_i(t_j) : i \in \N, j \in \{1, \dots, k\} \}
	$$
	are determinantal with kernel explicitly given by $K_\scrA(x, s; y, t) = H_\scrA+ J_\scrA$, with
	\begin{equation*}
	H_\scrA = \frac{-1_{s > t}}{\sqrt{4 \pi(s-t)}}e^{-\frac{(x-y)^2}{4 (s-t)}}, \quad  \quad J_\scrA= \frac{1}{(2\pi i)^2} \int_{\Ga_u} \int_{\Ga_v} \frac{G(x, s; v)}{G(y, t; u)} \frac{d v d u}{u - v}.
	\end{equation*}
	where
	\begin{equation*}
	G(x, s; u) = \exp \lf( ux + u^2s- \frac{1}{3}u^3 \rg).
	\end{equation*}
	Here the contour $\Ga_v$ goes from $e^{-i 2\pi/3}\infty$ to $0$ to $e^{i 2\pi/3} \infty$, and the contour $\Ga_u$ goes from $e^{-i \pi/3}\infty$ to $0$ to $e^{i \pi/3} \infty$.
	
	``Determinantal'' here means that for any subset of size $k$ of the random points $\scrA_i(t_j)$, the joint intensity $\rho_k$ of the point process is given by a $k\times k$ determinant of $K_\scrA$ according to
	\begin{equation*}
	    \rho_k\left( (x_1,t_1) , \ldots , (x_k,t_k) \right) = \det\left[ K_\scrA(x_i,t_i;x_j,t_j)\right]_{i,j=1}^k.
	\end{equation*}
	 See Chapter 4 of \cite{hough2009zeros} for an introduction to determinantal point processes and in particular Definition 4.2.1. for this definition in the general setting.   
\end{definition}
The fact that these two properties uniquely determine the parabolic Airy line ensemble is proven in Theorem 3.1 of  \cite{CH}.

We use the adjective ``parabolic'' to differentiate from the usual Airy line ensemble, the process $ \mathcal{A}(t) = \scrA(t) + t^2$, which is stationary in time. Note also that $\scrA$ has a flip symmetry: $\scrA(\cdot) \eqd \scrA(- \;\cdot)$.

We leave the discussion of the kernel formula for $\scrA$ to the end of the section as it is best motivated by first seeing the kernel for nonintersecting Bernoulli walks. For now, we continue with setting up the scaling under which nonintersecting Bernoulli walks converge to the parabolic Airy line ensemble.

As there is a symmetry between the top and bottom walks in an ensemble of $n$ nonintersecting Bernoulli walks, we will only analyze the bottom walks. For large $n$, the bottom walk concentrates around a deterministic `arctic' curve up to a lower order correction. The shape of the curve can be deduced from analyzing contour integral formulas (we will say a bit more about this in Section \ref{S:FDD}). The arctic curve $\ga_{n, \be}$ is given by the formula
\begin{equation}
\label{E:gamma-def}
\ga_{n, \be}(m) = \frac{(\sqrt{m \be} - \sqrt{n})^2}{1 + \be}\indic(m \be > n).
\end{equation}
The arctic curve $\ga =\ga_{n, \be}$ is constantly equal to $0$ for small $m$. This is the region where the higher Bernoulli walks have not yet moved to allow space for the bottom walk to start to move itself. For fixed $n$, in the limit $m \to \infty$ the slope of $\ga$ increases towards a limit of $\be/(1 + \be)$. This limit is the slope of an unconditioned Bernoulli walk. This property of the arctic curve is very natural: at large time scales, the walks spread further apart and so the nonintersecting condition is felt less and less.

Now let $\be_n \in (0, \infty), m_n \in \N$ be two sequences of real numbers with $m_n \be_n > n$ for all $n$ as in Theorem \ref{T:main-walk}. As in that theorem, we let $X_{n, i}, i \in \{1, \dots, n\}$ be $n$ nonintersecting Bernoulli walks with odds $\be_n$. We seek to derive scaling parameters $\chi_n$ and $\tau_n$ and a mean shift function $h_n$ so that
\begin{equation}
\label{E:ber-conv}
\chi_n^{-1}\lf(h_n - X_{n,i} \rg)(m_n + \floor{\tau_n t}) \Rightarrow \scrA_i(t)
\end{equation}
converges to the parabolic Airy line ensemble $\scrA$. To derive these parameters, we will use the Brownian Gibbs property of the parabolic Airy line ensemble, see \cite{CH}.
\begin{description}
\item{\textbf{Brownian Gibbs property:}}	For any $s < t$ and $k \in \N$, conditionally on the values of $\scrA_i(r)$ for $i \in \{1, \dots, k\}$ and $r \in \{s, t\}$ and the values of $\scrA_{k+1}(r)$ for $r \in [s, t]$, the Airy lines $\scrA_1 > \dots > \scrA_k$ restricted to the interval $[s, t]$ are given by $k$ Brownian bridges of variance $2$ between the appropriate endpoints, conditioned so that the lines remain nonintersecting.
\end{description}

Here when we say that a Brownian bridge (or Brownian motion) has variance $2$, we simply mean that its quadratic variation over any interval is equal to twice the length of that interval.

Nonintersecting Bernoulli walks satisfy a Gibbs property analogous to the Brownian Gibbs property of the parabolic Airy line ensemble (see Section \ref{S:Uniform} for details). For this Gibbs property to have any hope of surviving into the limit to give the Brownian Gibbs property, the shift $h_n$ needs to be linear; this is essentially due to the fact that the Brownian Gibbs property is preserved under linear shifts but not under shifts by any other function. We should therefore take $h_n$ to be the linear approximation of the arctic curve near $m_n$.

To see a limit which is locally Brownian with variance $2$, we also require that the spatial and temporal scaling factors $\chi_n$ and $\tau_n$ have the required relationship for random walks rescaling to variance $2$ Brownian motions. Near the point $m_n$, the slope of the bottom random walks is $\ga'=\ga_{n, \be_n}'(m_n)$. In a small local window, the walks do not feel the nonintersecting condition and look like unconditioned Bernoulli walks with this slope. For a Bernoulli walk with this slope to converge to Brownian motion with variance $2$, we require the scaling relationship
\begin{equation}
\label{E:bernoulli-scaling}
\ga'(1 - \ga') \tau_n = 2 \chi_n^2.
\end{equation}
The factor $\ga'(1 - \ga')$ is the effective variance of each step the lowest Bernoulli walks near the time $m_n$ (i.e. the variance of a Bernoulli random variable with mean $\ga'$). The scaling relationship \eqref{E:bernoulli-scaling} always needs to hold for any collection of nonintersecting random walks to converge to the parabolic Airy line ensemble, with the factor $\ga'(1- \ga')$ replaced by the effective variance of those random walks; several examples of this are contained in Section \ref{S:Corollaries}.

Finally, we need to scale so that the limit is stationary after the addition of a parabola. Since $h_n$ was given by the first order Taylor expansion of $\ga$ at $m_n$, the leading term in the difference in \eqref{E:ber-conv} is given by the second order term of the Taylor expansion of $\ga$ at $m_n$. In order to get the parabola $t^2$, we need the condition
\begin{equation}
\label{E:parabola-scale}
\frac{\ga''}2 \frac{\tau_n^2}{\chi_n}  = 1.
\end{equation}
The formulas \eqref{E:tau-xi-bern} are the unique solutions to \eqref{E:bernoulli-scaling} and \eqref{E:parabola-scale}. In the case of fixed $\be$ and $m_n = \al n$ for some $\al$, these two relationships give the usual KPZ scaling parameters of $\chi_n = c_1 n^{1/3}$ and $\tau_n = c_2 n^{2/3}$ for constants $c_1$ and $c_2$.

To prove that nonintersecting Bernoulli walks converge to the parabolic Airy line ensemble, we analyze determinantal formulas. We use a specialization of a formula from \cite{borodin2013markov}, Proposition 5.5, which is a reformulation of Theorem 2.25, Corollary 2.26 and Remark 2.27 in \cite{borodin2014anisotropic}. For any $\be, n$, the point process
$$
\{(X_{n, i}(t), t) : t \in \N, i \in \{1, \dots, n\} \}
$$
is determinantal on $\N^2$ with kernel
\begin{align}
\label{E:rw-kern}
K_{n, \be}(x, s; y, t) = H_{n, \be}(x, s; y, t) + J_{n, \be}(x, s; y, t),
\end{align}
where
\begin{align}
\label{E:rw-kern-2}
H_{n, \be} = - \indic_{s > t, x > y} \be^{x-y} \binom{s -t}{x -y} \quad \text{and} \quad
J_{n, \be} = \frac{1}{(2\pi i)^2}\int_{\Ga_w}\int_{\Ga_z} \frac{F_{n, \beta} (x, s ; w)}{F_{n, \beta}(y, t ; z)} \frac{1}{w(w-z)}\,dz\,dw.
\end{align}
Here
$$
F_{n, \be} (x, s; w) = (1 + \beta w)^s (1- w)^n w^{-x}.
$$
In the formulas above, the contours $\Ga_w$ and $\Ga_z$ are disjoint, oriented counterclockwise, and go around the poles at $0$ and $1$ respectively without encircling any other poles. Note that in \cite{borodin2013markov}, $H_{n, \be}$ is given as a contour integral which can be easily evaluated as the binomial coefficient \eqref{E:rw-kern-2} for the parameter regime we consider.

We will prove convergence of the kernel $K_{n, \be}$ to the kernel for the parabolic Airy line ensemble.
\cite{borodin2008asymptotics} give a contour integral formula for the kernel of stationary version of the Airy line ensemble, which we translate to our setting as follows.

\begin{lemma}
	\label{E:non-stat-kernel}
	The parabolic Airy line ensemble has kernel $K_\scrA(x, s; y, t) = H_\scrA+ J_\scrA$, with
	\begin{equation}
	\label{E:A-kernel}
	H_\scrA = \frac{-1_{s > t}}{\sqrt{4 \pi(s-t)}}e^{-\frac{(x-y)^2}{4 (s-t)}}, \quad  \quad J_\scrA= \frac{1}{(2\pi i)^2} \int_{\Ga_u} \int_{\Ga_v} \frac{G(x, s; v)}{G(y, t; u)} \frac{d v d u}{u - v}.
	\end{equation}
	where
	\begin{equation}
	\label{E:G-formula}
	G(x, s; u) = \exp \lf( ux + u^2s- \frac{1}{3}u^3 \rg).
	\end{equation}
	
	Here the contour $\Ga_v$ goes from $e^{-i 2\pi/3}\infty$ to $0$ to $e^{i 2\pi/3} \infty$, and the contour $\Ga_u$ goes from $e^{-i \pi/3}\infty$ to $0$ to $e^{i \pi/3} \infty$.
\end{lemma}

The Gaussian term in the kernel $K_\scrA$ suggests the locally Brownian behaviour (with variance $2$) of the Airy line ensemble. The kernel formula for $G$ is a manifestation of the KPZ $1:2:3$ scaling. The spatial parameter $x$ is paired with $u$, the time parameter gets paired with $u^2$, and there is a third $u^3$ term which can be thought of as having come from rescaling the number of lines $n$.

\begin{proof}
	We start with a formula for the kernel $K_{\mathcal{A}}({x}, {s}; {y}, {t})$ for the stationary Airy line ensemble ${\mathcal{A}}_i(t) = \scrA_i(t) + t^2$. This appears in \cite{borodin2008asymptotics} (see Proposition 4.8) and is based on a formula in \cite{johansson2003discrete} and \cite{prahofer2002scale}.
	\begin{equation}
	\label{E:stat-kernel}
	\begin{split}
	K_{\mathcal{A}}({x}, {s}; {y}, {t}) = \frac{-1_{{s} < {t}}}{\sqrt{4 \pi({t} - {s})}}\exp &\lf(-\frac{({y} -{x})^2}{4({t} -{s})} -\frac{1}2({t} - {s})({y} + {x}) + \frac{1}{12}({t} - {s})^3\rg) \\
	+ \frac{1}{(2\pi i)^2} \int_{\Ga'_u} \int_{\Ga'_v} &\exp \bigg({x}{s} - {y}{t} - \frac{1}3 {s}^3 + \frac{1}3 {t}^3 - ({x} - {s}^2) v + ({y} - {t}^2)u \\
	&- {s} v^2 + {t} u^2 + \frac{1}3 (v^3 - u^3) \bigg)\frac{dv du}{v - u}.
	\end{split}
	\end{equation}
	Here the contours in $u$ and $v$ are switched when compared with \eqref{E:A-kernel}. That is, $\Ga_u = \Ga'_v$ and $\Ga'_u  = \Ga_v$.
	Since $\scrA(t) = {\mathcal{A}}(t) - t^2$, we can express a kernel for $\scrA$ by changing coordinates and conjugating by a term of the form $f(x, s)g(y, t)$. Define
	$$
	\bar{K}_{\scrA}({x}, {s}; {y}, {t}) = \exp \lf(-({x} + \frac{2}3 s^2)s + ({y} + \frac{2}3 {t}^2) t\rg) K_{\mathcal{A}}({x} + {s}^2, {s}; {y} + {t}^2, {t}).
	$$
	After simplification, this gives
	$$
	\bar{K}_\scrA({x}, {s}; {y}, {t}) = \frac{-1_{{s} < {t}}}{\sqrt{4 \pi({t} - {\bs})}}e^{-\frac{({y} -{x})^2}{4({t} -{\bs})}} + \frac{1}{(2\pi i)^2} \int_{\Ga'_u} \int_{\Ga'_v} \frac{G({y}, {t}; u)}{G({x}, {s}; v)} \frac{dv du}{v - u}
	$$
	Here the contours are the same as in \eqref{E:stat-kernel}. Now by the flip symmetry $\scrA(\cdot) \eqd \scrA(-\;\cdot)$, we note that the kernel
	$$
	K_\scrA(x, s; y, t) = \bar{K}_\scrA({x}, -{s}; {y}, -{t}) = \frac{-1_{{\bs} > {t}}}{\sqrt{4 \pi({s} - t)}}e^{-\frac{({y} -{x})^2}{4({s} -t)}} + \frac{1}{(2\pi i)^2} \int_{\Ga_u'} \int_{\Ga_v'} \frac{G({y}, -{t}; u)}{G({x}, -s; v)} \frac{dv du}{v - u},
	$$
	also gives the parabolic Airy line ensemble. Making the change of variables $v \mapsto -v$ and $u \mapsto -u$ in the above integral then gives the representation in the lemma. Note that when we do this, the contours in $u$ and $v$ switch.
\end{proof}

\section{Kernel convergence for Bernoulli walks}\label{S:FDD}

In this section we prove the preliminary version of Theorem \ref{T:main-walk}, convergence of finite dimensional distributions. The proof of Theorem \ref{T:main-walk} is then completed in Section \ref{S:Uniform}.

Convergence of finite dimensional dimensional distributions follows from appropriately strong convergence of $K_{n, \be_n}$ to $K_\scrA$ after rescaling and conjugation. Throughout this section will simplify notation along the lines of $K=K_n = K_{n, \be_n}$ depending on context.

Convergence of the binomial term $H_n$ is easier to understand probabilistically, so we will start there. This will reveal the necessary conjugation of the kernel. In this section $\ga_n, \ga_n', \ga_n''$ refer to $\ga_{n, \be_n}$ and its derivatives evaluated at the point $m_n$. We have the following translation between the unscaled parameters $x_n, s_n, y_n, t_n$ in the prelimit and their limiting versions $x, s, y, t$:
\begin{align}
\label{E:xnsn}
        &s_n = m_n + \floor{\tau_n s}, \qquad  &&t_n = m_n + \floor{\tau_n t}, \qquad \\
        \nonumber
&x_n = \floor{\ga_n + \ga_n' \tau_n s - \chi_n x}, \qquad &&y_n = \floor{\ga_n + \ga_n' \tau_n t - \chi_n y}
\end{align}

In this scaling,
$$
H_n(x_n, s_n; y_n, t_n) = - 1_{s_n > t_n, x_n > y_n} \be_n^{x_n-y_n} \binom{\tau_n (s-t) + \ep_1}{\ga_n' \tau_n(s-t) +  \chi_n(x-y) + \ep_2}.
$$
Here $\ep_1, \ep_2 \in [-2, 2]$ are negligible error terms coming from the floors in $x_n, s_n, y_n, t_n$.
The binomial factor in $H_n$ already suggests (i.e.\ by the de Moivre-Laplace central limit theorem) that under some rescaling, $H_n$ should converge to a Gaussian kernel. Moreover, $\chi_n$ and $\tau_n$ are already set-up to be the correct spatial and temporal rescalings for a Bernoulli walk with slope $\ga_n'$ to converge to Brownian motion with variance $2$, see \eqref{E:bernoulli-scaling}. Using this picture, we see that if we set
$\de_n$ so that
\begin{equation} \label{E:delta_def}
\frac{\de_n \be_n}{1 + \de_n \be_n} = \ga_n',
\end{equation}
then we observe that
\begin{equation}
\label{E:de-n-H}
\frac{\de_n^{x_n - y_n} }{(1 + \be_n \de_n)^{s_n-t_n}} H_{n},
\end{equation}
is the probability that a Bernoulli random walk with slope $\ga'$ is at location $x - y$ after $s-t$ steps. We call $\de_n \in (0, 1)$ the \textbf{damping parameter} for the Bernoulli random walks. It represents how much the bottom walk feels the conditioning from the walks above; the equation relating $\de_n$ to $\ga_n'$ shows that the bottom walk effectively behaves like a Bernoulli walk of odds $\de_n \be_n$ near the time $m_n$. After rescaling up by the spatial scaling $\chi_n$, \eqref{E:de-n-H} converges to the Gaussian term in $K_\scrA$ by the central limit theorem.

This analysis reveals the correct conjugation needed for $K_{n,\be}$ to converge to the Airy kernel. This conjugation could also be obtained by analyzing the second term $J_n$. The standard strategy for proving convergence of terms of this type is to search for a double critical point $w_n$ of the function $\log F_n(\ga_n, m_n; \cdot)$ and then perform a steepest descent analysis around this double critical point. The conjugation should then be given by rescaling the integrand in \eqref{E:rw-kern-2} so that it always equals $1$ at this critical point.

 The arctic curve $\ga_n$ can also be identified by double critical point considerations. For a particular value of $m_n$, two choices of $\ga_n$ result in a function with a double critical point, while others will yield two single critical points. These choices are the arctic curves for the highest and lowest walks.

A calculation reveals that the double critical point for $\log F_n(\ga_n, m_n; w)$ happens exactly at the damping parameter $w = \de_n$ \eqref{E:delta_def}. The appropriate conjugation is therefore
\begin{equation}
\label{E:F-formula}
\frac{F_n(y, t; \de_n)}{F_n(x, s; \de_n)} = \frac{\de_n^{x - y}}{(1 + \be_n \de_n)^{s-t}},
\end{equation}
which is the same conjugation as in equation \eqref{E:de-n-H}.
We can now precisely state the kernel convergence.

\begin{prop}
	\label{P:kern-cvg}
	Let $\beta_n, m_n$ be two sequences of real numbers so that statement (ii) in Theorem \ref{T:main-walk} is satisfied. With the scaling $x_n, y_n, s_n, t_n$ as in \eqref{E:xnsn}, define the conjugated and rescaled version of the random walk kernel $K_n$ by
	$$
	\tilde{K}_{n}(x, s; y, t) =  \chi_n \frac{\de_n^{x_n - y_n} }{(1 + \be_n \de_n)^{s_n-t_n}}K_n(x_n, s_n; y_n, t_n) = \chi_n \frac{F_n(y_n, t_n; \de_n)}{F_n(x_n, s_n; \de_n)}K_n(x_n, s_n; y_n, t_n)
	$$
	Then as functions from $\R^4 \to \R$, we have that $\tilde{K}_{n} = K_\scrA + o_n,$ where the error term $o_n$ is small in the following sense:
	\begin{enumerate}[label=(\roman*)]
		\item For any $s, t \in \R$ and a compact set $D \sset \R^2$, we have that
		$$
		\lim_{n \to \infty} \sup_{x, y \in D} |o_n(x, s; y, t)|= 0.
		$$
		\item For any $s, t, b \in \R$, there exist constants $c, d> 0$ such that $|o_n(x, s; y, t)| \le ce^{-d(x + y)}$ for all $n \in \N$, and $x, y \ge b$.
	\end{enumerate}
\end{prop}
\medskip

\noindent \emph{Proof of the convergence of finite dimensional distributions in Theorem \ref{T:main-walk} assuming Proposition \ref{P:kern-cvg}.} We first assume that $\chi_n \to \infty$ with $n$. Let $\scrA^n_i(t) = \chi_n^{-1}\lf(h_n - X_{n,i} \rg)(m_n + \floor{\tau_n t})$ denote the $i$th rescaled walk at time $t$, the left hand side of \eqref{E:ber-conv}. 

The random functions $\scrA^n_i$ take values in the space of functions from $[-m_n, \infty) \to \R$. For any fixed time $t$, the point $\scrA^n_i(t)$ is contained in the set $\Z_{n, t} := \chi_n^{-1}(\Z + h_n(m_n + \floor{\tau_n t}))$. For any fixed collection of times $\mathbf{t} = (t_1, \dots, t_k)$, the collection of points $(\scrA^n_i(t_j), t_j), i \in \{1, \dots, n\}, j \in \{1, \dots, k\}$ is a determinantal point process on the space 
$$
S_{n, \bf t} = (\Z_{n, t_1} \X \{t_1\}) \cup \dots \cup (\Z_{n, t_k} \X \{t_k\})
$$
with kernel $\tilde K_n$ with respect to counting measure on $S_n$ rescaled by $\chi_n^{-1}$. The rescaling of the background counting measure accounts for the factor of $\chi_n$ in the definition of $\tilde K_n$. Note that the rescaled counting measure on $S_{n, \bf t}$ converges vaguely to Lebesgue measure on $\R \X \{t_1, \dots, t_k\}$ (see Chapter 4 of \cite{kallenberg2017random} for the definition of vague convergence).

Now, to prove convergence of finite dimensional distributions, we need to show that
$$
\scrA^n_i(t_j) \Rightarrow \scrA_i(t_j)
$$
jointly over $i \in \N$ and $j \in \{1, \dots, k\}$ for any $\mathbf{t}$. For this, we just need that
\begin{enumerate}[label=(\roman*), nosep]
    \item The point measures $P^n_1, \ldots, P^n_k$, where
    $$
    P^n_j = \sum_{i=1}^n \delta_{\scrA^n_i(t_j)}
    $$
    converge jointly in distribution with respect to the vague (also called the local weak) topology to their limit $P_j$ defined similarly in terms of $\scrA$.
    \item For each $j \in \{1, \dots, k\}$, we have
    $$
    \lim_{a \to \infty} \limsup_{n \to \infty} \expt P^n_j[a, \infty) = 0.
    $$
\end{enumerate}
A simple way to see that (i) and (ii) implies the convergence of finite dimensional distributions is through the Skorokhod coupling, Theorem 3.30 in \cite{kallenberg2006foundations}. Let $X_j^n=\mathcal L^n_1(t_j)^+$. By (ii) the random variables $X_j^n$ are tight, so for every subsequence $J_0$ of $\mathbb N$ there is a further subsequence $J$ so that $X_j^n, P_j^n$ converge jointly in law. By Skorokhod coupling, this implies that there is coupling in which this convergence is almost sure. Note that we are able to apply Skorokhod coupling since the vague topology is Polish, see Theorem 4.2 in \cite{kallenberg2017random}. This implies that along $J$,  the order statistics of each $P^n_j$ converge almost surely to the order statistics of $P_j$ since convergence of $X_j^n$ prevents points from escaping to $+\infty$. That is, $\mathcal L_i^n(t_j)\to \mathcal L_i(t_j)$ for all $i \in \N$. Since $J_0$ was arbitrary, the distributional convergence follows. 



To prove (i), we will show that for any finite set of intervals $[a_i, b_i] \sset \R,$ and indices $j(i) \in \{1, \dots, k\}$, that
\begin{equation}
\label{E:prod-Pn}
\expt \left[\prod_{i=1}^m P^n_{j(i)}[a_i, b_i] \right] \to \expt \lf[\prod_{i=1}^m P_{j(i)}[a_i, b_i] \rg].
\end{equation}
Before proving \eqref{E:prod-Pn}, let us first explain why this gives (i).
Equation \eqref{E:prod-Pn} implies vague tightness of the random measures $P_j$ by Theorem 4.10 in \cite{kallenberg2017random}. Moreover, the moments of any limit point of $(P^n_1, \dots, P^n_k)$ must match the right side of \eqref{E:prod-Pn}. We must show that any limit point $Q = (Q_1, \dots, Q_k)$ is equal to $(P_1, \dots, P_k)$. Indeed, $Q$ is determined by all the joint distributions of the form
$$
Q_{\mathbf{i, a, b}} := (Q_{i(1)}[a_1, b_1], \dots, Q_{i(m)}[a_m, b_m]).
$$
The distribution of such a random vector is determined by its moments as long as
the $n$th moments $\mu_n$ of all individual coordinates satisfy $\limsup (\mu_n)^{1/n}/n < \infty$ (e.g. see Theorem 3.3.11 in \cite{durrett2019probability} for the single variable case, and Theorem 2 in \cite{kleiber2013multivariate} for the multivariate extension). In our setting, this limsup condition follows from the fact that each $P_j$ is a determinantal point process with a Hermitian kernel, and the fact the number of points in any set in such a process has exponential tails, see \cite{hough2009zeros}, Lemma 4.2.6.








Now, the left hand side of \eqref{E:prod-Pn} can be written as a finite linear combination of integrals of the form
\begin{equation}
\label{E:intensities}
\int_{\prod_{i=1}^{m'} [c_i, d_i]} \det \lf(\tilde K_n(x_k, t_{j'(k)}; x_\ell, t_{j'(\ell)}) \rg)_{k, \ell \in [1, m']} d\nu_{n, j'(1)}(x_1) \dots d\nu_{n, j'(m')}(x_{m'}).
\end{equation}
Here each of the intervals $[c_i, d_i]$ is equal to one of the intervals $[a_i, b_i]$, each $j'(\ell) \in \{1, \dots, k\}$, $m' \in \{1, \dots, m\}$, and for each $i$, the measure $\nu_{n, i}$ is simply counting measure on $\Z_{n, t_i}$ rescaled by $\chi_n^{-1}$. (See \cite{hough2009zeros}, Remark 1.2.3 for details on how to construct the linear combination; an explicit example is given there for the case when $m = 2$. Essentially, the linear combination arises from writing the moments of the process in terms of the factorial moments of the process, which can be done by inverting a particular upper triangular matrix. Factorial moments are  naturally expressed in terms of integrals of determinants as in \eqref{E:intensities}.) For each $i$, the measure $\nu_{n, i}$ converges vaguely to Lebesgue measure on $\R$ as $n \to \infty$.  Uniform-on-compact convergence of the kernels $\tilde K_n$, Proposition \ref{P:kern-cvg} (i), then implies \eqref{E:prod-Pn}.

For item (ii), note that
$$
\expt P^n_j[a, \infty) = \int_a^\infty \tilde K_n(t_j,x,t_j,x) d\mu_{n, j}(x).
$$
This is bounded uniformly in $\N$ by $ce^{-d a}$ by Proposition \ref{P:kern-cvg} (ii).

For the other direction of Theorem \ref{T:main-walk}, if $\chi_n$ does not approach $\infty$ with $n$, then there is a constant $b > 0$ and a subsequence $J \sset \N$ such that $\chi_n \le b$ for all $n \in J$. Now, $\mathcal L^n_1(0) = - \chi_n^{-1} X_{n, i}(m_n) \in \chi_n^{-1} \Z$, so deterministically along this subsequence, $\mathcal L^n_1(0) \notin (0, b^{-1})$. On the other hand, $\mathcal L_1(0)$ is a Tracy-Widom random variable, whose law is mutually absolutely continuous with respect to Lebesgue measure, so $\mathcal L^n_1(0) \not \cvgd \mathcal L_1(0)$. Therefore $\scrL^n$ cannot converge to the parabolic Airy line ensemble.
\qed

\subsection*{Proof of Proposition \ref{P:kern-cvg}}

The main part of the proof of Proposition \ref{P:kern-cvg} involves deforming the $w$ and $z$ contours for $J_n$ from \eqref{E:rw-kern-2} so that they look like the Airy contours around the double critical point $\de_n$ for $\log F_n(\ga_n, m_n; \cdot)$, and then performing a steepest descent analysis to show that the contribution to the contours away from the double critical point is negligible. 

The main difficulty in doing this is in constructing the appropriate contours. Because of the $(w-z)^{-1}$ term in formula \eqref{E:rw-kern-2} for $J_n$, we will need the contours to be sufficiently separated. When $\be_n$ and $m_n/n$ are fixed as $n \to \infty$, this is guaranteed along the true steepest ascent/descent contour for $\log F_n(\ga_n, m_n; \cdot)$, but it is more difficult to guarantee this when $m_n/n$ and $\be_n$ vary with $n$. Also, we need the function $\log F_n(x, m_n; \cdot)$ to behave well along the contours even when $x$ is much less than $\ga_n$ in order to guarantee Proposition \ref{P:kern-cvg} (ii).

Because of these difficulties, we resort to constructing the contours with an implicit geometric construction. We will also bound the relevant functions along these contours with geometric arguments. While it may be possible to proceed in a more direct and computational manner, we did not find a way to do this is in an elegant way.

The following propositions construct appropriate contours. 
Define
\begin{equation}
\label{E:L-def}
L(w)=L_{\al, \be}(w) = \log (1 - w) + \al \log(w + \be^{-1}) - \frac{(\al \be - 1)^2}{1 + \be} \log (w).
\end{equation}

Observe that $L_{m_n/n, \be_n} = n^{-1} \log F_n(\ga_n, m_n; \cdot)$. $L'$, the derivative of $L$, is a rational function whose directions of descent and ascent can be analyzed by geometric considerations. To simplify notation in this proposition and the next one, we will write
\begin{equation}
\label{E:Lde}
\de = \frac{\sqrt{\al \be} - 1}{\sqrt{\al \be} + \be}, \quad \mathand \qquad \tilde{\ga} = \frac{(\sqrt{\al \be} - 1)^2}{1 + \be}.
\end{equation}
When $\alpha = m_n/n$ and $\beta = \beta_n$ then this definition of $\de \in (0, 1)$ agrees with \eqref{E:delta_def}.
\begin{figure}
	\begin{center}
\begin{tikzpicture}[scale=2.2,
point/.style = {circle, draw=black, inner sep=0.1cm,
	color=black,fill, scale=0.5,
	node contents={}},
every label/.append style = {font=\scriptsize}]


\def\angle{60}
\def\cos{0.5}
\def\sin{0.866}
\def\R{0.866}

\draw [help lines,->] (-2, 0) -- (2,0);
\draw [help lines,->] (0, 0) -- (0,1.6);

\draw[line width=1pt,   decoration={ markings,
	mark=at position 0.5 with {\arrow[line width=1.2pt]{>}}},
postaction={decorate}] (1,0) -- (1-\cos*\R,\sin*\R);
\draw[line width=1pt,   decoration={ markings,
	mark=at position 0.5 with {\arrow[line width=1.2pt]{>}}},
postaction={decorate}] (1-\cos*\R,\sin*\R) arc (55:180:0.9401);

\draw[dashed] (1,0) arc (0:180:1);
\node at (-0.5,0)[point, label=below:$-\be^{-1}$];
\node at (0,0)[point, label=below:$0$];
\node at (1.2,0)[point, label=above:$\delta$];
\node at (1.75,0)[point, label=below:$1$];
\node at (1,0)[scale=0.5,circle,node contents={},draw=black,label=below:$\delta - \eta$];
\node at (-1,0)[scale=0.5,circle,node contents={},draw=black,label=below:$-\delta+\eta$];

\node at (2,-0.2){$x$};
\node at (-0.24,1.6) {$iy$};

\end{tikzpicture}
\end{center}
\caption{The contour $\scrC_w$ in Proposition \ref{P:G1-bounds-w} for positive times. It starts at a point $\de - \eta$ for some small $\eta$ and stays within a circle of radius $\de - \eta$ about the origin. Moreover, $\Re(L)$ descends proportional to $L'$ along the entire contour. We first follow a straight line emanating from the point $\de - \eta$ and then append a circular arc about the origin. The point at which $\scrC_w$ switches from following a straight line to a circular arc is chosen so that $\scrC_w$ always stays away from the point $-\be^{-1}$.}	\label{fig:cw}
\end{figure}

\begin{prop}
	\label{P:G1-bounds-w} Let $\de, \tilde \ga$ be as in \eqref{E:Lde}.
	There exist universal constants $c_1, c_2 > 0$ such that the following holds for all choices of parameters $\al, \beta > 0$ with $\al \be > 1$, and for every $\eta \le \de/24$. There exists $t_w \in (0, \pi \de]$ and a contour $\scrC_w:[-t_w, t_w] \to \C$ which is parametrized by arc length and has the following properties:
	\begin{enumerate}[label=(\Roman*)]
		\item $\scrC_w(t) = \close{\scrC_w(-t)}$ for all $t \in [0, t_w]$.
		\item $\scrC_w(0) = \de - \eta$, $\scrC_w(t_w) \in (-\infty, 0)$, and
		$$
		\scrC_w(t) \in \{z: \Im z > 0, \; \de/4 \le |z| \le \de - \eta\} \qquad \forall t \in (0, t_w).
		$$
		\item For $t \in [0, \de/4]$, we have that $\scrC_w(t) = \de - \eta + te^{2\pi i/3}$.
		\item The following bounds holds for all $t \in [-t_w, t_w]$:
		\begin{align*}
		\Re(L(\scrC_w(t))) &\leq \Re(L(\de)) + \frac{c_1(\al + 1 - \tilde{\ga})}{(\de + \be^{-1})\de(1-\de)}\eta^3 - \int_0^{|t|} c_2|L'(\scrC_w(s))|ds. \\
		\nonumber
		|L'(\scrC_w(t))| &\ge  \frac{c_2(\al + 1 - \tilde{\ga})t^2}{(\de+ \beta^{-1})\de(1-\de + |t|)}.
		\end{align*}
	\end{enumerate}
\end{prop}

The main consideration driving the proof is as follows: by the simple form of $L'$, we can always locate the directions in which $\Re(L)$ is decreasing at a point $w \in \C$ by looking at a particular sum of angles formed by $w$ and the points $\be^{-1}, 0, \de,$ and $1$. We can use this to create a contour $\mathcal{C}_w$ which is the union of a linear piece and a circular arc, along which the behaviour of $L'$ can be  controlled by simple geometric arguments.
Throughout the proof, all contours will be parametrized by arc length.


\begin{proof} We will only construct $\scrC_w$ for positive times and then extend it to all times by setting $\scrC_w(t) = \close{\scrC_w(-t)}$. With this choice, the bounds in point (IV) will automatically hold for negative times since $\Re (L(z)) = \Re (L(\bar z))$ and point (I) is immediate. 
	
\textbf{Step 1: Constructing $\scrC_w$.} \qquad We will define the curve  $\scrC_w$ piecewise, see Figure \ref{fig:cw} for the basic idea. Let $\eta \le \de/24$ and let $\theta \in \{\pi/6, \pi/5\}$; we will choose $\theta$ later on in the construction. Let $t_0 > 0$ be the unique time such that 
$$
\Arg(\de-\eta + e^{2\pi i/3}t_0) = \theta,
$$ 
and set $\scrC_w(t) = \de-\eta + e^{2\pi i/3}t$ for $t \in [0, t_0]$. Now for $t \ge t_0$, define $\scrC_w$ so that it that traverses the circle $\{ z : |z| = |\scrC_w(t_0)|\}$ counterclockwise around the origin. Let $t_w > t_0$ be the time when $\scrC_w$ hits the real axis. Finally, choose $\theta$ so that the quantity
\begin{equation}
\label{E:t00}
\inf_{t \in [0, t_w]} |\scrC_w(t) + \be^{-1}|
\end{equation}
is maximized. Using the sine law, it is easy to check that $t_0 \in [\de/4, \de]$ and that (II, III) hold in this construction regardless of $\theta, \eta$. Moreover, our flexibility in choosing $\theta$ guarantees that \eqref{E:t00} is bounded below by $c \de$ for a universal constant $c$.
 
\textbf{Step 2: Verifying (IV).} \qquad	We first compute
	\begin{equation}
	\label{E:L-prime}
	L'(w) = \frac{(\al + 1 - \tilde{\ga})(w - \de)^2}{(w + \beta^{-1})(w-1)w}.
	\end{equation}
	The constant factor $\al + 1 - \tilde{\ga} > 0$, so we can write
	\begin{equation}
	\label{E:G1-expand}
	\Arg(L'(w)) = 2\Arg(w - \de) - \Arg(w - 1) - \Arg(w) - \Arg(w - \be^{-1})
	\end{equation}
	Hence at a point $w$ in the upper half plane, the direction of steepest descent for $\Re(L)$ is given by
	\begin{equation}
	\label{E:L-arg-prime}
	\pi - \Arg(L'(w)) = \pi + \Arg(w - 1) + \Arg(w) + \Arg(w + \be^{-1}) - 2\Arg(w - \de).
	\end{equation}
	First, we prove the following upper bound on the angle between $\Arg(\scrC_w'(t))$ and $\pi - \Arg(L'(w))$ for all $t \in [3\eta, t_0]$: 
	\begin{equation}
	\label{E:5pi12}
	|\pi - \Arg(L'(\scrC_w(t))) - \Arg(\scrC_w'(t))| \le 5\pi/12.
	\end{equation}
	This will show that $\Re(L)$ is decreasing at a rate of at least a rate of at least $\cos(5\pi/12)|L'(\scrC_w(t))|$ for $t \in [3\eta, t_0]$.
	 Noting that $\Arg(\scrC_w(t))$ is increasing in $t$ on the interval $[0, t_0]$, we have the following inequality chain in that interval:
	\begin{equation}
	\label{E:ineq-arg-1}
	0 \le \Arg(\scrC_w(t) + \be^{-1}) \le \Arg(\scrC_w(t)) \le \theta.
	\end{equation}
	We also have that
	\begin{equation}
	\label{E:ineq-arg-2}
2\pi/3 \le \Arg(\scrC_w(t) - \de) \le \Arg(\scrC_w(t) - 1) \le \pi.
	\end{equation}
	Moreover, for $t_0 \ge t \ge 3 \eta$, the sine law gives that $\Arg(\scrC_w(t) - \de) \le 3\pi/4$. Putting this together with \eqref{E:ineq-arg-1} and \eqref{E:ineq-arg-2} and plugging the bounds into \eqref{E:L-arg-prime}, we have
	\begin{equation*}
	\pi - \Arg(L'(\scrC_w(t))) \in [\pi - 3\pi/4, 2\pi - 2(2\pi/3)+ 2\theta] \qquad \mathforall t \in [t_0/2, t_0],
	\end{equation*}
	which implies \eqref{E:5pi12}.
	
	Next, we claim that $\Re(L)$ is nonincreasing along $\scrC_w$ in the interval $[t_0, t_w)$. To see this, first observe that $\Arg(\scrC_w'(t)) = \Arg(\scrC_w(t)) + \pi/2$ for $t \in [t_0, t_w)$, so by \eqref{E:L-arg-prime},
	\begin{equation*}
	\pi - \Arg(L'(\scrC_w(t))) - \Arg(\scrC_w'(t)) =  \pi/2 + \Arg(\scrC_w(t) - 1) + \Arg(\scrC_w(t) + \be^{-1}) - 2\Arg(\scrC_w(t) - \de).
	\end{equation*}
	To show $\Re(L)$ is nonincreasing, we just need to show that the right hand side above is in the interval $[-\pi/2, \pi/2]$. Let $A(z) := \pi - \Arg(z)$ be the angle formed by the ray to a point $z$ and the negative real axis. Then we can rewrite the right hand side above as
	\begin{equation}
	\label{E:crux-ineq}
	-\pi/2  + \Arg(\scrC_w(t) + \be^{-1}) + 2A(\scrC_w(t) - \de) - A(\scrC_w(t) - 1).
	\end{equation}
	Noting that $A(\scrC_w(t) - \de) > A(\scrC_w(t) - 1)$ and that $\Arg(\scrC_w(t) + \be^{-1})$ and  $A(\scrC_w(t) - \de)$ are nonnegative, the quantity \eqref{E:crux-ineq} is always strictly bounded below by $-\pi/2$. To get an upper bound for \eqref{E:crux-ineq}, observe that
	\begin{equation}
	\label{E:be-w}
	\Arg(\scrC_w(t) + \be^{-1}) + 2A(\scrC_w(t) - \de) \le \pi \quad \text{if and only if} \quad |\scrC_w(t) + \be^{-1}| \le |\de + \be^{-1}|.
	\end{equation}
	The reason for this is purely geometric: if the right side of \eqref{E:be-w} holds, then in the triangle formed by the three points $\de, -\be^{-1},$ and $\scrC_w(t)$, the angle at $\scrC_w(t)$ will be greater than or equal to $A(\scrC_w(t) - \de)$ and vice versa.
	
	We now bound $\Arg(\scrC_w(t) + \be^{-1}) + 2A(\scrC_w(t) - \de)$ by verifying the right side of \eqref{E:be-w}. Observe that
	$$
	|\de + \be^{-1}| - |\scrC_w(t) + \be^{-1}| \ge |\de| - |\scrC_w(t)|.
	$$
	The right hand side above is bounded below by $\eta$ by (II). Therefore $\Arg(\scrC_w(t) - \be^{-1}) + 2A(\scrC_w(t) - \de) \le \pi$, and so since $A(C_w - 1) \ge 0$, we have that \eqref{E:crux-ineq} is bounded above by $\pi/2$, as desired.
	
	Next we bound the magnitude of $L'$ along $\scrC_w(t)$. The construction of the contour guarantees that
	\begin{align*}
	|\scrC_w(t) - \de| \sim \eta + t, \quad, |\scrC_w(t)| \sim \de, \quad |\scrC_w(t)-1| \sim 1 - \de + t, \quad |\scrC_w(t)+ \be^{-1}| \sim \de  + \be^{-1},
	\end{align*}
	where the $\sim$ equivalence notation means that the ratio of the two sides is bounded above and below by positive constants that are independent of all parameters. It is easiest to see these equivalences by studying Figure \ref{fig:cw}. For the final equivalence above, we have combined the observations that \eqref{E:t00} is bounded below by $c \de$ and that $|\scrC_w(t)| \sim \de$.
	Therefore along $\scrC_w$ we have that
	$$
|L'(\scrC_w(t))| \sim \frac{(\al + 1 - \tilde \ga)(\eta + t)^2}{(\de+ \beta^{-1})\de(1-\de + t)}.
	$$
	Condition (IV) then follows by combining the following facts:
\begin{itemize}[nosep]
	    \item $\Re(L(\scrC_w(t)))$ is decreasing on the interval $[3 \eta, t_0]$ at a rate of at least $\cos(5\pi/12)|L'(\scrC_w(t))|$.
	    \item $\Re(L(\scrC_w(t)))$ is nonincreasing on the interval $[t_0, t_w]$.
	    \item $3 \eta \le t_0/2$ and $t_w \sim t_0 \sim \de$.
	\end{itemize}
The second term in the first inequality in (IV) comes from the possibility that $\Re(L(\scrC_w(t)))$ is increasing for $t \in [0, 3\eta]$ and the final term in that inequality comes from the good rate of decrease of $\Re(L(\scrC_w(t)))$ in the interval $[3 \eta, t_0]$, which is comparable in size to the entire interval $[0, t_w]$.
\end{proof}

We now prove an analogous proposition for the $z$-contour.

\begin{figure}
\begin{center}
	\begin{tikzpicture}[scale=1.9,
	point/.style = {circle, draw=black, inner sep=0.1cm,
		color=black,fill, scale=0.5,
		node contents={}},
	every label/.append style = {font=\scriptsize}]
	
	
	\def\angle{80}
	\def\cos{0.1743}
	\def\sin{0.984}
	\def\R{0.866}
	\def\coss{-0.1743}
	\def\sinn{0.984}
	\def\RR{0.866}

	\draw [help lines,->] (-2, 0) -- (3,0);
	\draw [help lines,->] (0, 0) -- (0, 2);
	
	\draw[line width=1pt,   decoration={ markings,
		mark=at position 0.5 with {\arrow[line width=1.2pt]{>}}},
	postaction={decorate}] (1.4,0) -- (1.4+\cos*\R,\sin*\R);
	
	\draw[line width=1pt,   decoration={ markings,
		mark=at position 0.5 with {\arrow[line width=1.2pt]{>}}},
	postaction={decorate}] (1.4+\cos*\R,\sin*\R) -- (1.4+\cos*\R+\coss*\RR,\sin*\RR + \sinn*\RR);

	\draw[line width=1pt,   decoration={ markings,
		mark=at position 0.5 with {\arrow[line width=1.2pt]{>}}},
	postaction={decorate}]  (1.4+\cos*\R,\sin*\R) arc (193-\angle:4:1);

	\draw[dashed] (1.4,0) arc (0:180:1.4);
	
	\node at (-0.85,0)[point, label=above:$-\be^{-1}$];
	\node at (0.1,0)[point, label=above:$0$];
	\node at (1.2,0)[point, label=above:$\delta$];
	\node at (1.4,0)[scale=0.5,circle,node contents={},draw=black,label=below:$\delta + \eta$];
	\node at (-1.4,0)[scale=0.5,circle,node contents={},draw=black,label=below:$-\delta-\eta$];
	\node at (2,0)[point, label=above:$1$];
	
	\node at (3,-0.2){$x$};
	\node at (-0.24,2) {$iy$};
	
	\end{tikzpicture}
\end{center}
\caption{A sketch of possibilities for the contour $\scrC_z$ in Proposition \ref{P:G1-bounds-z}. The contour starts $\de + \eta$ for some small $\eta$, stays outside of a circle of radius $\de + \eta$ about the origin, and $\Re(L)$ ascends proportionally to  $L'$ along the entire contour. It is a piecewise construction whereby that first follows a straight line from the point $\de - \eta$. If the directional derivative of $\Re(L)$ becomes too small at some point, then we turn either left along another straight line, or right along a circle centered at $1$. One of these choices guarantees that $\Re(L)$ ascends at a fast enough rate. If $\scrC_z$ turns to the right, then $t_z < \infty$; otherwise, $t_z = \infty$.}
\label{fig:cz}
\end{figure}

\begin{prop}
	\label{P:G1-bounds-z} Let $\de, \tilde \ga$ be as in \eqref{E:Lde}.
	There exist universal constants $c_1, c_2 > 0$ such that the following holds for all choices of parameters $\al$ and $\beta$ with $\al \be > 1$, and for every $\eta \le [\de \wedge (1-\de)]/20$. There exists $t_z \in (0, \infty]$ and a contour $\scrC_z:[-t_z, t_z] \to \C$ which is parametrized by arc length and has the following properties:
	\begin{enumerate}[label=(\Roman*)]
		\item $\scrC_z(t) = \close{\scrC_z(-t)}$.
		\item $\scrC_z(0) = \de + \eta$, $\scrC_z(t_z) \in (1, \infty)$ when $t_z < \infty$, and
		$$
		\scrC_z(t) \in \{z: \Im z > 0, |z| > \max(\de + \eta, c_2 t)\}, \qquad \forall t \in (0,t_z).
		$$
		\item For $t \in [0, (\de \wedge (1-\de))/6]$, we have that $\scrC_z(t) = \de + \eta + te^{4\pi i/9}$.
		\item The following bounds holds for all $t \in [-t_z, t_z]$:
		\begin{align}
		\label{E:Z-G-bd-prop-deriv}
		\Re(L(\scrC_z(t))) &\geq \Re(L(\de)) - \frac{c_1(\al + 1 - \tilde \ga)}{(\de + \be^{-1})\de(1-\de)}\eta^3 + \int_0^{|t|} c_2|L'(\scrC_z(t))|dt. \\
		\nonumber
		|L'(\scrC_z(t))| &\ge  \frac{c_2(\al + 1 - \tilde \ga)t^2}{(\de+ \beta^{-1} + |t|)(\de + |t|)(1-\de + |t|)}.
		\end{align}
	\end{enumerate}
\end{prop}

For constructing the $z$-contour, our goal is to have the contour follow a direction of ascent for $\Re(L)$, rather than a direction of descent. We  do this by ensuring that $-\Arg(L'(\scrC_z(t)))$ and $\Arg(\scrC_z'(t))$ are close. Throughout the proof, all contours are parametrized by length. We encourage the reader to look at Figure \ref{fig:cz} before reading the proof to aid in the understanding of the contour construction.


\begin{proof}
	We will only construct $\scrC_z$ for positive $t$, and then extend by the formula $\scrC_z(t) = \close{\scrC_z(-t)}$. With this choice, the bounds in point (IV) will automatically hold for negative times since $\Re (L(z)) = \Re (L(\bar z))$ and point (I) is immediate.

	\textbf{Step 1: The first segment of $\scrC_z$:} \qquad Let $\eta \le [\de \wedge (1-\de)]/20$. Define $\scrC_z^*(t) = \de + \eta + te^{4\pi i/9}$. The true contour $\scrC_z$ will equal $\scrC_z^*$ for small $t$, to be made precise as follows. 
	Define
	\begin{equation}
	\label{E:arg-z-cont}
	t_0 = \inf \lf\{t \ge 6 \eta :-\Arg(L'(\scrC_z^*(t))- \Arg( \bar \scrC'_z(t))\notin (-4\pi/9, 4\pi/9) \rg\},
	\end{equation}
	and set $\scrC_z(t) = \scrC_z^*(t)$ for $t \le t_0$. We first claim that $t_0 > 6 \eta$. Indeed, by the sine law we have
	\begin{equation}
	\label{E:C1-bd}
	\Arg(\scrC_z^*(t) - \de) \in [7\pi  /18, 8\pi /18] \qquad \text{ for } t \ge 6 \eta 
	\end{equation}
	and by the sine and cosine laws and the upper bound on $\eta$ we have
	\begin{equation*}
	\Arg(\scrC_z^*(6 \eta) - 1) \in [8\pi/9, \pi], \qquad \Arg(\scrC_z^*(6 \eta)), \Arg(\scrC_z^*(6 \eta) - \be^{-1})) \in [0, \pi/9]. 
	\end{equation*}
	Combining these facts with equation \eqref{E:G1-expand} implies that $-\Arg(L'(\scrC_z^*(t)) \in (0, 8\pi/9)$, and hence $t_0 > 6 \eta$. Note that we may have $t_0 = \infty$. 
	
	Now, the definition of $t_0$ gives that
	\begin{equation}
	\label{E:L-low}
	\frac{d}{dt} \Re(L(\scrC_z(t)) \ge \cos(4\pi/9) |L'(\scrC_z(t))|, \qquad \text{for } t \in [6 \eta, t_0].
	\end{equation}
	Also, along the contour $\scrC_z$, by the formula \eqref{E:L-prime} and basic geometric considerations we have the estimate
	\begin{equation}
	\label{E:G1-zestimate}	|L'(\scrC_z(t))| \sim \frac{(\al + 1 - \tilde \ga)(t+ \eta)^2}{(\de + \be^{-1} + t)(\de + t)(1 - \de + t)}.
	\end{equation}
	This estimate combined with \eqref{E:L-low} yields conditions (III) and (IV) in the proposition for $t < t_0$. Since conditions (I) and (II) are also satisfied, this completes the proof of the proposition when $t_0 = \infty$.
	
	\textbf{Step 2: Extending the contour for $t > t_0$ when $t_0 < \infty$.} Since $t_0 \ne 6 \eta$, by continuity we must have that $-\Arg(L'(\scrC_z(t_0))- \Arg(\scrC_z'(t_0)) = \pm 4\pi/9$, or in other words $-\Arg(L'(\scrC_z(t_0)) \in \{0, 8 \pi/9\}$.
	
	\noindent \textbf{Case 1: $
	-\Arg(L'(\scrC_z(t_0)) = 0.
	$} 
\noindent For $t > t_0$, define $\scrC_z$ so that it traverses the circle $\{z : |z - 1| = |\scrC_z(t_0)|\}$ clockwise. Let $t_z$ be the time when $\scrC_z$ hits the real axis. 
	
	Expanding out $\Arg(L'(\scrC_z(t_0))$ using equation \eqref{E:G1-expand} and the bounds in \eqref{E:C1-bd}, we get that
	\begin{equation}
	\Arg(\scrC_z(t_0)) + \Arg(\scrC_z(t_0) + \be^{-1}) + \Arg(\scrC_z(t_0) -1) \in [7\pi/9, 8\pi/9].
	\end{equation}
	In particular, this implies that
	\begin{equation}
	\label{E:1-angle}
	0 \le\Arg(\scrC_z(t_0) -1) \le 8\pi/9.
	\end{equation}
	Now, using \eqref{E:1-angle}, the fact that $1 - (\de + \eta) \ge (1- \de)/2$, and the sine law, we have $t_0 \ge (1-\de)/6$, giving (III).
	
	Next, using \eqref{E:1-angle} again we have that
	$$
	\Arg(\scrC_z'(t)) = -\pi/2 + \Arg(\scrC_z(t) - 1) \in [-\pi/2, 7 \pi/18]
	$$
	for $t \in [t_0, t_z)$. Using this, the fact that $\Arg(\scrC_z(t_0)-\de - \eta) = 4\pi/9$, and the fact that $\scrC_z(t)$ remains in the upper half plane for $t \in [t_0, t_z)$ implies that
	\begin{equation}
	\label{E:C-2g-bd}
	0 \le \Arg(\scrC_z(t)-\de - \eta) \le 4\pi/9, \qquad t \in [0, t_z],
	\end{equation}
	and so $|\scrC_z(t)| \ge \de + \eta$ for all $t$. Finally, $|\scrC_z(t)|/t$ is bounded below by an absolute positive constant by construction, yielding (II).
	
To prove (IV), we first show that $\Re(L)$ is increasing along $\scrC_z$ whenever $t \in [t_0, t_z)$. The difference between the steepest ascent direction for $L$ and the direction of $\scrC_z$ is given by
	$$
	-\Arg(L'(\scrC_z(t)) - \Arg(\scrC_z'(t)) = \Arg(\scrC_z(t)) + \Arg(\scrC_z(t) + \be^{-1}) - 2 \Arg(\scrC_z(t)-\de) + \pi/2.
	$$
	Here we have used \eqref{E:G1-expand} and the fact that $\Arg(\scrC_z'(t)) = \Arg(\scrC_z(t) - 1) - \pi/2$. To bound the right hand side above, we use the chain of inequalities
	\begin{equation}
	\label{E:standard-ineqs-z}
	0 < \Arg(\scrC_z(t) + \be^{-1}) < \Arg(\scrC_z(t)) < 	\Arg(\scrC_z(t)-\de) < \Arg(\scrC_z(t)-\de - \eta) 
	\end{equation}
	for $t \in [t_0, t_z)$ along with the bound \eqref{E:C-2g-bd} to get that
	$$
	-\Arg(L'(\scrC_z(t)) - \Arg(C_z'(t))\in (-7\pi/18, \pi/2),
	$$
	so $\Re (L')$ is decreasing along $\scrC_z$ for $t \in [t_0, t_z)$. 
	
	Now we put everything together to extend (IV) to all $t \in [t_0, t_z]$. First, we have
	\begin{equation}
	\label{E:tzt0}
	t_z \le  10 t_0 \le 20(t_0 - 6 \eta),
	\end{equation}
	where the first inequality follows from \eqref{E:1-angle}, and the second inequality follows from the upper bound on $\eta$ and the lower bound on $t_0$ in (III). Next, by basic geometric considerations $|L'(C_z(t))| \sim |L'(C_z(t_0))|$ for all $t \in [t_0, t_z]$. Since $t_z \le 10 t_0$, \eqref{E:G1-zestimate} extends to all $t \in [0, t_z]$ by possibly decreasing $c_2$, giving the second bound in (IV). 
	
	To extend the first bound in (IV) to all $t \in [t_0, t_z]$, we use that $t_z \sim t_0 - 6 \eta$, and that $|L'(C_z(t))| \sim |L'(C_z(t_0))|$ and  $\Re (L(\scrC_z(t))$ is increasing for $t \in [t_0, t_z]$.

	\noindent {\bf Case 2:}	$
	-\Arg(L'(\scrC_z(t_0))= 8\pi/9.
	$
	In this case, we will finish the contour by defining  $t_z = \infty$ and setting
	$$
	\scrC_z(t) = \scrC_z(t_0) + (t-t_0)e^{5\pi i/9}
	$$
	on the interval $[t_0, \infty)$. 
	
	\noindent Expanding out $\Arg(L'(\scrC_z(t_0))$ using equation \eqref{E:G1-expand} and the bounds in \eqref{E:C1-bd}, we get that
	\begin{equation*}		
	\Arg(\scrC_z(t_0)) + \Arg(\scrC_z(t_0) -1) + \Arg(\scrC_z(t_0) + \be^{-1}) \in [15 \pi /9, 16\pi/9].
	\end{equation*}
	Since $\Arg(\scrC(t_0) -1) \in [0, \pi]$, this gives that
	\begin{equation}
	\label{E:2pi6theta}
	\Arg(\scrC_z(t_0)) + \Arg(\scrC_z(t_0) + \be^{-1}) \ge 2\pi/3.
	\end{equation}
	Since $\Arg(\scrC_z(t_0)) > \Arg(\scrC_z(t_0) + \be^{-1})$, we have $\Arg(\scrC_z(t_0)) > \pi/3$. This, along with the construction of the contour implies that $|\scrC_z(t)| \ge \de + \eta$ everywhere. Also, $|\scrC_z(t)|/t$ is bounded below by an absolute constant, yielding (II).	Moreover, since $\Arg(\scrC_z(t_0)) > \pi/3$ the sine law implies that $t_0 > \de$, giving (III).

	The proof of (IV) for $t > t_0$ will be similar to the $t_0 = \infty$ case. Observe that \eqref{E:G1-zestimate} still holds in the present setting so we just need to establish \eqref{E:L-low} for $t > t_0$, or equivalently that
	\begin{equation}
	\label{E:ArgArg}
-4\pi/9 \le - \Arg(L'(\scrC(t))) - \Arg(\scrC_z'(t)) \le 4\pi/9
	\end{equation}
	For the upper bound in \eqref{E:ArgArg}, for $t \ge t_0$ we have that
	\begin{align*}
	 - \Arg&(L'(\scrC(t))) - \Arg(\scrC_z'(t)) \\
	=& \Arg(\scrC_z(t)) + \Arg(\scrC_z(t) + \be^{-1}) + \Arg(\scrC_z(t) -1) - 2\Arg(\scrC_z(t) - \de) - \frac{5\pi}{9} \\
	\le&  \Arg(\scrC_z(t) -1) - \frac{5\pi}{9} \\
	\le& 4\pi/9.
	\end{align*}
	The first inequality above follows from the bounds in \eqref{E:standard-ineqs-z}, which also hold in this case. For the lower bound, note that
	\begin{itemize}[nosep]
		\item $\Arg(\scrC_z(t) -1) \ge \Arg(\scrC_z(t) - \de)$,
	    \item 	$\Arg(\scrC_z(t) - \de) \le 5\pi/9$ for all $t$,
	    \item  $\Arg(\scrC_z(t))+ \Arg(\scrC_z(t) + \be^{-1})$ is an increasing function of $t$, and is hence always bounded below by $2\pi/3$ by \eqref{E:2pi6theta}.
	\end{itemize}
	Combining these bounds, we get that
		\begin{align*}
	- \Arg&(L'(\scrC(t))) - \Arg(\scrC_z'(t)) \\
	=& \Arg(\scrC_z(t)) + \Arg(\scrC_z(t) + \be^{-1}) + \Arg(\scrC_z(t) -1) - 2\Arg(\scrC_z(t) - \de) - \frac{5\pi}{9} \\
	\ge&  \Arg(\scrC_z(t)) + \Arg(\scrC_z(t) - \be^{-1}) - \Arg(\scrC_z(t) - \de) - \frac{5\pi}{9} \\
	\ge& -4\pi/9. \qquad \qquad \qquad \qquad \qquad \qquad \qquad \qquad \qedhere
	\end{align*}
\end{proof}

We are now almost ready to prove Proposition \ref{P:kern-cvg}. 
 Recall the definitions of the scalings $x_n, s_n$ from the beginning of Section \ref{S:FDD}. We use the decomposition
\begin{equation}
\label{E:F-be-n}
\log F_n(x_n, s_n; w) = s_n \log \be_n + nL(w) + \tau_n s L_t(w) + \chi_n x L_x(w),
\end{equation}
where $L = L_{m_n/n, \be_n}$ is as in \eqref{E:L-def}, and
\begin{align}
\label{E:Lt_Lx_defn}
	L_t(w) &:= \log (\be_n^{-1} + w) - \ga_n'(m_n) \log w \qquad \mathand \qquad L_x(w)  := \log w.
\end{align}
There is an implicit dependence on $n$ in $L_t$ that will be suppressed throughout the proof. After deforming the contours for $J_n$, all the weight will come from a region of size $O(\rho_n^{-1})$  around the double critical point $\de_n$, where
\begin{equation}
\label{E:rho-n}
\rho_n := \chi_n/\de_n.
\end{equation}

Near $\de_n$, we will pick up the first non-trivial Taylor expansion term in each of $L, L_t,$ and $L_x$: these become the $u^3, u^2$, and $u$ terms respectively in the limiting integrand $G$, see \eqref{E:G-formula}. This will be made precise in the forthcoming proof of Proposition \ref{P:kern-cvg}.

 In order to guarantee that the error terms in our asymptotics drop away, we need to show that after rescaling the complex plane by $\rho^{-1}_n$, the distance from the critical point $\de$ to each of the distinguished points $0, 1,$ and $-\be^{-1}$ goes to $\infty$ with $n$.

\begin{lemma}
	\label{L:necess-scale}
	Let $m_n, \be_n$ be sequences with $m_n \be_n > n$ such that the spatial scaling parameter $\chi_n \to \infty$ as $n \to \infty$. Then as $n \to \infty$, we have that
	\begin{equation}
	\label{E:scalings}
	\de_n \rho_n \to \infty, \qquad (\de_n + \be_n^{-1})\rho_n \to \infty, \qquad \mathand \quad (1 - \de_n)\rho_n \to \infty.
	\end{equation}
\end{lemma}

\begin{proof}
	Since $\de_n \rho_n = \chi_n$, and since $\be_n > 0$, the first two convergences are immediate. It remains to prove the third convergence. For readability of the formulas, we will write $\al = m_n/n, \be = \be_n$ and write $\ga', \ga''$ for the derivatives of the arctic curve evaluated at $m_n$. We can expand out the spatial scaling parameter $\chi_n$ using the formula \eqref{E:tau-xi-bern} in Theorem \ref{T:main-walk} as:
\begin{equation}
\label{E:zenx}
\chi_n = \lf( \frac{[\ga'(1 - \ga')]^2}{2 \ga''}\rg)^{1/3} =
\lf(\frac{n\sqrt{\be}(\sqrt{\al} + \sqrt{\be})^2 (\sqrt{\al\be} - 1)^2}{\sqrt{\al} (1 + \be)^3}\rg)^{1/3}.
\end{equation}
Using \eqref{E:Lde}, we have
\begin{equation}
[(1- \de_n)\rho_n]^3 = \lf(\frac{\chi_n (1 - \de_n)}{\de_n}\rg)^3 = \frac{n\sqrt{\be}(\sqrt{\al} + \sqrt{\be})^2}{\sqrt{\al} (\sqrt{\al \be} -1)}\ge n,
\end{equation}
and so the left hand side approaches infinity with $n$.
\end{proof}

The intuition behind the three poles at $0, 1,$ and $-\be^{-1}$ is that after the appropriate rescaling, the distance from the critical point to each pole stands as a proxy for a particular scaling parameter going to $\infty$. The pole at $1$ represents the number of lines and comes from the $n$ term in the definition of $F_n$, the pole at $\be^{-1}$ represents the time scaling and comes from the $t$ term, and the pole at $0$ represents the spatial scaling and comes with the $x$ term. In the case of the pole at $0$, this is a very precise statement, since the distance to that pole after rescaling is simply $\chi_n$.

\begin{proof}[Proof of Proposition \ref{P:kern-cvg}]
	We can write $\tilde K_n = \tilde H_n + \tilde J_n$, where $\tilde H_n$ and $\tilde J_n$ are rescaled and coordinate-changed versions of $H_n$ and $J_n$. Showing that $\tilde H_n$ converges to the corresponding term in $K_\scrA$ pointwise follows from the central limit theorem for Bernoulli walks, see the discussion before Proposition \ref{P:kern-cvg}. Showing this with the desired error bound and follows from a quantitative version of the central limit theorem for Bernoulli walks (e.g. an application of Stirling's formula \cite{durrett2019probability}, Section 3.1). We move on to deal with the more complicated term $J_n$.
	
	For ease of notation during the rest of the proof, we will omit from our notation the dependence of parameters on $n$, e.g. $\de = \de_n, \be = \be_n$.
	
	\noindent \textbf{Step 1: Comparing $L, L_t$, and $L_x$.} \qquad In Propositions \ref{P:G1-bounds-w} and \ref{P:G1-bounds-z} we defined contours that behave well with the function $L$. However, as can be seen from \eqref{E:F-be-n}, our integrand also has $L_t$ and $L_x$ terms. Our first aim is to show that these terms are negligible away from the critical point when compared with the $L$ term, justifying our choice of contours.
	
	 Recalling the computation of  $L'$ and computing the derivatives of $L_t$ and $L_x$ (recall their definition from \eqref{E:Lt_Lx_defn}) gives
	\begin{align}
	\label{E:L-ga}
	nL'(w) &= \frac{(m_n + n - \ga)(w - \de)^2}{w(w + \be^{-1})(w-1)}, \qquad
	L_t'(w) = \frac{(1-\ga')(w-\de)}{w(w + \be^{-1})}, \quad \mathand \quad
	L_x'(w) = \frac{1}w.
	\end{align}
	Here $\ga$ and $\ga'$ are the values of the arctic curve at the point $m_n$.
	Computations using the above formulas, and the definitions \eqref{E:tau-xi-bern}, \eqref{E:Lt_Lx_defn}, and \eqref{E:rho-n}, show that the scaling parameters satisfy
	\begin{equation}
	\label{E:scale-taylor}
	n = -\frac{2 \rho_n^3}{L'''(\de)}, \qquad \tau_n = \frac{2\rho_n^2}{L_t''(\de)} \qquad \mathand \quad \chi_n = \frac{\rho_n}{L_x'(\de)}.
	\end{equation}
	These equations reveal that the first three terms of the Taylor series expansion of $F_{n}$ locally looks like $G$ around the double critical point $\de$ in the right scaling regime. Using these expressions in conjunction with the relationships in \eqref{E:L-ga}, we get that
	\begin{align}
	\lf|\frac{nL'(w)}{\tau_n L_t'(w)}\rg| &=  \lf|\frac{\rho_n (1-\de)(w - \de)}{2(w - 1)}\rg| \ge\frac{\rho_n |1-\de||w - \de|}{2(|w - \de| +|1-\de|)} \\
	\lf|\frac{\tau_n L_t'(w)}{\chi_n L_x'(w)}\rg| &=  \lf|\frac{\rho_n (\de + \be^{-1})(w - \de)}{w + \be^{-1}}\rg| \ge\frac{\rho_n |\de + \be^{-1}||w - \de|}{|w - \de| +|\de + \be^{-1}|}.
	\end{align}
	Both of the right hand sides above are increasing in $|w - \de|$. Moreover, by Lemma \ref{L:necess-scale} we have that $\rho_n^{-1} |1 - \de|, \rho_n^{-1}|\de + \be^{-1}| \to 0$ with $n$. In particular, this implies that for any fixed $\Omega > 0$ and $|w - \de| \ge \rho_n^{-1} \Omega$, and for all large enough $n$ we have
	\begin{equation}
	\label{E:small-deriv-bds}
	\lf|\frac{nL'(w)}{\tau_n L_t'(w)}\rg| \ge \Omega/3, \qquad \mathand \qquad \lf|\frac{\tau_n L_t'(w)}{\chi_n L_x'(w)}\rg| \ge \Omega/3.
	\end{equation}
	

\noindent \textbf{Step 2: Deforming the contours.} \qquad	
 First deform the contours for $\tilde J_{n}$ to the contours from Propositions \ref{P:G1-bounds-w} and \ref{P:G1-bounds-z} so that $\Ga_w$ becomes $\scrC_w$ and $\Ga_z$ becomes $\scrC^-_z$ with the parameter $\eta$ in that lemma equal to $\rho_n^{-1}$. Here $\scrC^-_z$ is equal to $\scrC^-_z$ but with the opposite orientation. Since $\rho_n^{-1} = o(\min (\de, 1- \de))$ by Lemma \ref{L:necess-scale}, these contours will satisfy the assumptions of those propositions for large enough $n$.
	
Note that $\scrC^-_z$ may go to $\infty$ rather than forming a closed loop around $1$. To justify this deformation, observe that for any $t, b \in \R$, and for all large enough $n \in \N$, the following holds for all $y \ge b$ and $w \in \R$:
	$$
	\int_{|z| = r} \lf|\frac{1}{F(y_n, t_n; z)(w - z)} \rg| dz \le (c_\be r)^{y_n - t_n - n} \to 0 \qquad \mathas r \to \infty.
	$$
	Here $c_\be$ is a positive $\be$-dependent constant. 
	For the convergence above, we have used that 
	$$
	y_n - t_n \le \floor{\ga(m_n) + \ga'(m_n) \tau_n t - \chi b} - \floor{m_n  + \tau_n t}.
	$$
	The right hand side above converges to $-\infty$ with $n$ since $\ga(m_n) \le m_n$ and $(1-\ga'(m_n))\tau_n \gg \chi_n$ by \eqref{E:bernoulli-scaling} and since $\chi_n \to \infty$. 
	
	Now for each $\Om > 0$, we will write $\scrC_w = \scrC_{w, \Om} \cup \scrC_{w, \Om}^c$, where $\scrC_{w, \Om}$ is the restriction of $\scrC_w(t)$ to the interval $t \in [-\Om \rho_n^{-1}, \Om \rho_n^{-1}]$ and $\scrC^c_{w, \Om}$ is the remaining part of $\scrC_w$. We similarly decompose $\scrC^-_z = \scrC^-_{z, \Om} \cup \scrC_{z,\Om}^{-,c}$. Note that for any fixed $\Om$, for large enough $n$, the contour $\scrC_{w, \Om}$ consists of two rays emanating from $\de - \rho_n^{-1}$ with arguments $2\pi/3$ and $-2\pi/3$ by Proposition \ref{P:G1-bounds-w}(III) and the fact that $\rho_n^{-1} = o(\min (\de, 1- \de))$. Similarly, the contour $\scrC^-_{z,\Om}$ consists of two rays emanating from $\de + \rho_n^{-1}$ with arguments $4\pi/9$ and $-4\pi/9$ for large enough $n$ by Proposition \ref{P:G1-bounds-z}(III).
	
	\noindent \textbf{Step 3: Convergence along $\scrC_{w, \Om}, \scrC^-_{z, \Om}$.} \qquad Taylor expanding $L, L_t,$ and $L_x$ around the point $\de$ gives
	\begin{equation}
	\label{E:G-expands}
	\begin{split}
	nL(\de + \rho_n^{-1}u) - nL(\de)&= -u^3/3 + O(n u^4\rho_n^{-4} L^{(4)}(\de))\\
	\tau_n L_t(\de + \rho_n^{-1}u) - \tau_n L_t(\de)&= u^2 + O(\tau_n u^3\rho_n^{-3} L_t^{(3)}(\de)) \\
	\chi_n L_x(\de + \rho_n^{-1}u) - \chi_n L_x(\de) &= u + O(\chi_n u^2\rho_n^{-2} L_x^{(2)}(\de))
	\end{split}
	\end{equation}
	for $u \in [-\Om, \Om]$. 
	Here the constant in the big-$O$ notation depends only on $\Om$.
	To deal with the error terms, observe that
	\begin{align*}
	L^{(4)}(\de) &= -L^{(3)}(\de) \lf(\frac{1}{\de} + \frac{1}{\de + \be^{-1}} + \frac{1}{\de - 1} \rg) \\
	L_t^{(3)}(\de) &= -L_t^{(2)}(\de) \lf(\frac{1}{\de} + \frac{1}{\de + \be^{-1}} \rg) \\
	L_x^{(2)}(\de) &= -L_x^{(1)}(\de) \lf(\frac{1}{\de} \rg)
	\end{align*}
	By these calculations, the equations in \eqref{E:scale-taylor}, and the scaling relationship in Lemma \ref{L:necess-scale}, each of the errors in \eqref{E:G-expands} tends to $0$ as $n \to \infty$, uniformly over $u \in [-\Om, \Om]$. Therefore making the change of variables $u = (w-\de)\rho_n $ and $v = (z-\de)\rho_n $, we can write
	\begin{align}
	\label{E:J-trunc}
	\chi_n &\frac{F(y_n, t_n; \de)}{F(x_n, s_n; \de)} \frac{1}{(2\pi i)^2}\int_{\scrC_{w, \Om}}\int_{\scrC^-_{z, \Om}} \frac{F(x_n, s_n ; w)}{F(y_n, t_n ; z)} \frac{1}{w(w-z)}\,dz\,dw \\
	\nonumber
	&= \frac{\chi_n}{(2\pi i)^2}\int_{\Ga^u_\Om}\int_{\Ga^v_\Om} [1 + \ep(x, s, y, t; u, v)]\frac{G (x, s ; w)}{G(y, t; z)} \frac{1}{(u-v)} \frac{1}{\rho_n \de + u}\,du\,dv.
	\end{align}
	Here the error term $\ep(x, s, y, t; u, v)$ comes from the error terms in \eqref{E:G-expands}. In particular, it converges to $0$ uniformly for $u, v \in [-\Om, \Om]$ for all fixed $x, s, y, t$ by the discussion above. Moreover, by \eqref{E:F-be-n} it scales at most linearly in $x, s, y, t$. Therefore $\ep \to 0$ uniformly for bounded values of $x, s, y, t$ and $u, v \in [-\Om, \Om]$.
	
	The contours $\Ga^u_\Om$ and $\Ga^v_\Om$ are the rescaled versions of $\scrC_{w, \Om}$ and $\scrC_{z, \Om}$. We can write $\Ga^u_\Om$ explicitly as consisting of two rays emanating from $-1$ of length $\Om$, with arguments $2\pi/3$ and $-2\pi/3$, and we can similarly write  $\Ga^v_\Om$ explicitly as consisting of two rays emanating from $1$ of length $\Om$, with arguments $4\pi/9$ and $-4\pi/9$. Because $\chi_n = \rho_n \de$, and $\rho_n \de \to \infty$, we can then conclude that by the bounded convergence theorem, the right hand side of \eqref{E:J-trunc} converges to
	\begin{equation} \label{eq:Gint}
	\frac{1}{(2\pi i)^2} \int_{\Ga^u_\Om} \int_{\Ga^v_\Om} \frac{G (x, s ; w)}{G(y, t ; z)} \frac{1}{(u-v)} du \, dv
	\end{equation}
	uniformly on compact sets of the parameters $x, s, y, t$. Moreover, using the explicit formula \eqref{E:G-formula} for $G$, and using that $G$ has no poles or zeros we can 
	see that for any compact set $K \sset \R^4$, there exist constants $c_1> 0$, $a\in(0,1)$ such that for all $\Om > 0$, we have
		\begin{equation*}
	\sup_{(x, s; y, t) \in K} \lf|\frac{1}{(2\pi i)^2} \int_{\Ga^u_\Om} \int_{\Ga^v_\Om} \frac{G (x, s ; w)}{G(y, t ; z)} \frac{dv du}{(u-v)} - \frac{1}{(2\pi i)^2} \int_{\Ga_u} \int_{\Ga_v} \frac{G(x, s; v)}{G(y, t; u)} \frac{d v d u}{u - v} \rg| \leq c_1 a^{\Omega^3}.
	\end{equation*}
	Here $\Gamma_u$ and $\Gamma_v$ are the contours for the Airy kernel as in \eqref{E:G-formula}. Putting this together with the convergence of \eqref{E:J-trunc} to \eqref{eq:Gint} gives that the limsup as $n \to \infty$ of
	\begin{equation*}\sup_{(x, s; y, t) \in K} \lf| \frac{F (y_n, t_n; \de)}{F(x_n, s_n; \de)} \frac{\chi_n}{(2\pi i)^2}\intop_{\scrC_{w, \Omega}}\intop_{\scrC^-_{z, \Omega}} \frac{F(x_n, s_n ; w)}{F(y_n, t_n ; z)} \frac{\,dz\,dw}{w(w-z)} - \frac{1}{(2\pi i)^2} \int_{\Ga_u} \int_{\Ga_v} \frac{G(x, s; v)}{G(y, t; u)} \frac{d v d u}{u - v}  \rg| 
	\end{equation*}
	is at most $c_1 a^{\Om^3}$.
	
	\noindent \textbf{Step 4: Bounds along $\scrC_{w, \Om}^c, \scrC_{z, \Om}^{-, c}$.} \qquad 
	 To complete the proof of Proposition \ref{P:kern-cvg}(i), we just need to show that for every compact $K \sset \R^4$, we have
	\begin{align}
	\label{E:C-Z}
	\lim_{\Omega \to \infty} \limsup_{n \to \infty} \sup_{(x, s; y, t) \in K} \lf| \chi_n\frac{F (y_n, t_n; \de)}{F(x_n, s_n; \de)} \frac{1}{(2\pi i)^2}\int_{\scrC^c_{w,\Omega}}\int_{\scrC^-_z} \frac{F(x_n, s_n ; w)}{F(y_n, t_n ; z)} \frac{\,dz\,dw}{w(w-z)} \rg|  = 0
	\end{align}
	and similarly with $\scrC_w$ in place of $\scrC^c_{w, \Om}$ and $\scrC^{-, c}_{z,\Om}$ in place of $\scrC_z^-$. By combining the estimates in \eqref{E:small-deriv-bds} with those in Propositions \ref{P:G1-bounds-w}(IV) and \ref{P:G1-bounds-z}(IV) , we have that for every compact set $K$, there exist universal constants $c_1$ and $c_2$ such that for large enough $n$, the following bound holds along $\scrC_w$ for $x, s, y, t \in K$.
	\begin{align*}
 &\;\Re(\log F (x_n, s_n; \scrC_w(t))) \\
	 \leq&\;  \Re(\log F (x_n, s_n;\de)) + \frac{c_1(m_n + n - \ga(m_n))}{(\de + \be^{-1})\de(1-\de)}\rho_n^{-3} - \int_0^{|t|} c_2 n |L'(\scrC_w(s))|ds \nonumber \\
	\nonumber
	=&\; \Re(\log F(x_n, s_n;\de)) + c_1/2 - \int_0^{|t|} c_2 n |L'(\scrC_w(s))|ds.
	\end{align*}
	Here we have used that $2\rho_n^3 =- n L'''(\de)$ to go from the second to the third line.
	Similarly, along $\scrC^-_z$ we have that
		\begin{align*}
	\Re(\log F(x_n, s_n ; \scrC_z(t))) &\geq  \Re(\log F(x_n, s_n;\de)) - c_1/2 + \int_0^{|t|} c_2 n |L'(\scrC_z(t))|dt.
	\end{align*}
 We now parametrize the contours so that $\scrC_w$ gets parametrized by $t_1 \in [-t_w, t_w]$ and $\scrC^-_z$ gets parametrized by $t_2\in [-t_z, t_z]$ as in Propositions \ref{P:G1-bounds-w} and \ref{P:G1-bounds-z}, and then make the substitution $r_1 = \rho_n t_1$ and $r_2 = \rho_n t_2$. Noting that Propositions \ref{P:G1-bounds-w}(II) and \ref{P:G1-bounds-z}(II) imply that $|\scrC_{w}(t) - \scrC_z(s)| \ge \rho_n^{-1}$ and $|\scrC_{w}(t)| \ge \de/4$, we have the following upper bound on the supremum in \eqref{E:C-Z}:
	
	\begin{align}
	\nonumber
	&\chi_n \int_{[-t_w, t_w] \smin [- \Omega\rho_n^{-1}, \Omega \rho_n^{-1}]} \int_{-t_z}^{t_z}  \frac{4e^{c_1} dt_1 dt_2}{\de\rho_n^{-1}} \exp \lf(-c_2n \lf( \int_0^{|t_1|}|L'(\scrC_w(s))|ds + \int_0^{|t_2|} |L'(\scrC^-_z(s))|ds\rg) \rg)  \\
	\label{E:integral-to-bd}
	&= \int_{[-t_w \rho_n, t_w \rho_n] \smin [-\Omega, \Omega]} \int_{-t_z \rho_n}^{t_z \rho_n} 4e^{c_1} d r_1 dr_2 \exp\lf(-c_2n \lf( \int_0^{\rho_n^{-1} |r_1|}|L'(\scrC_w(s))|ds + \int_0^{\rho_n^{-1} |r_2|} |L'(\scrC^-_z(s))|ds\rg) \rg).
	\end{align}
	For each of the integrated terms in the exponential, we have the following bound after a change of variables by using the second inequality in Propositions \ref{P:G1-bounds-w}(IV) and \ref{P:G1-bounds-z}(IV). Here $\scrC$ is $\scrC^-_z$ or $\scrC_w$.
	\begin{align}
	\nonumber
	\int_0^{\rho_n^{-1} r}n|L'(\scrC(t))|dt &= \int_0^{r} \rho_n^{-1} n|L'(\scrC(\rho_n^{-1} s))| ds \\
	\nonumber
	&\ge \int_0^{r} \frac{(m_n + n - \ga(m_n))\rho_n^{-3} s^2}{(\de + \rho_n^{-1} s)(\de + \be^{-1} + \rho_n^{-1} s)(1 - \de + \rho_n^{-1} s)} ds \\
	\nonumber
	&= \int_0^{r} \frac{1}{2} \frac{\de s}{\de + \rho_n^{-1} s} \frac{(1 -\de) s}{1 - \de + \rho_n^{-1} s} \frac{\de + \be^{-1}}{\de + \be^{-1} + \rho_n^{-1} s} ds\\
	\label{E:RHS-final}
	&\ge \int_0^{r} \frac{1}{16} (s \wedge \rho_n \de)(s \wedge \rho_n (1-\de))\lf(1 \wedge \frac{\rho_n (\de+\be^{-1})}s\rg)ds.
	\end{align}
	For the equality in the third line, we have used \eqref{E:scale-taylor}. For the final equality, we have thrice used the bound $a/(b+ c) \ge [(a/b) \wedge (a/c)]/2$ for positive real numbers $a, b, c$.
	For large enough $n$, all of $\rho_n \de, \rho_n(1-\de)$ and $\rho_n(\de + \be^{-1})$ are strictly greater than $1$ by Lemma \ref{L:necess-scale}. Therefore the integrand above is bounded below by
	\begin{align}
	\label{E:60}
	\frac{1}{16} \lf(s^2 \wedge 1 \wedge\frac{\rho_n^3 \de (1 - \de)(\de+\be^{-1})}s\rg) \ge \frac{1}{16} \lf(s^2 \wedge 1 \wedge\frac{n}s\rg).
	\end{align}
	The inequality in \eqref{E:60} follows since $\rho_n^3 \de (1 - \de)(\de+\be^{-1}) = 2(m_n + n - \ga(m_n)) > n$. Using \eqref{E:60} and \eqref{E:RHS-final} implies that for $\Om \ge 0, |r_1| \ge \Om$ and $r_2 \in \R$, for all large enough $n$ we have
\begin{align*}
\exp&\lf(-c_2n \lf( \int_0^{\rho_n^{-1} r_1}|L'(\scrC_w(s))|ds + \int_0^{\rho_n^{-1} r_2} |L'(\scrC^-_z(s))|ds\rg) \rg) \\
&\le ce^{- c_2(r_1 \wedge n)/32}(1 \wedge r_1^{-10}) (1 \wedge r_2^{-10}) \\
&\le ce^{-c_2 \Om/32} (1 \wedge r_1^{-10}) (1 \wedge r_2^{-10}).
\end{align*}
Here $c$ is a large constant that may change from line to line.
Integrating out $r_1, r_2$ then gives that \eqref{E:integral-to-bd} is bounded above by $ce^{-c_2 \Om/32}$ for all large enough $n$, and hence \eqref{E:C-Z} holds. Moreover, the exact same arguments work to show that \eqref{E:C-Z} with $\scrC_w$ in place of $\scrC_{w, \Omega}^c$ and $\scrC^{-, c}_{z, \Omega}$ in place of $\scrC^-_z$ since we only used integrand bounds which are the same along the two contours $\scrC^-_z$ and $\scrC_w$. This completes the proof of Proposition \ref{P:kern-cvg} (i).
	
	\noindent \textbf{Step 5: Proposition \ref{P:kern-cvg} (ii).} \qquad
	For Proposition \ref{P:kern-cvg} (ii), observe that we can bound $\tilde J_n$  by comparing to the case when $x=0, y=0$ by using the identity $F(x, s; w) = F(y, s, w) w^{y - x}$. Indeed, letting $0_n$ denote $x_n$ or $y_n$ when $x=0$ or $y=0$ we have
	\begin{align*}
\bigg|\chi_n &\frac{F(y_n, t_n; \de)}{F(x_n, s_n; \de)} \frac{1}{(2\pi i)^2}\int_{\scrC_w}\int_{\scrC^-_z} \frac{F(x_n, s_n; w)}{F(y_n, t_n ; z)} \frac{dz dw}{w(w-z)}\bigg| \\
= \bigg|\chi_n &\frac{F(0_n, t_n; \de)}{F(0_n, s_n; \de)} \delta^{\chi_n(y - x)} \frac{1}{(2\pi i)^2}\int_{\scrC_w}\int_{\scrC^-_z} \frac{F(0_n, s_n; w)}{F(0_n, t_n ; z)} w^{\chi_n x} z^{-\chi_n y} \frac{dz dw}{w(w-z)}\bigg| \\
&\le \lf(\frac{\de - \rho_n^{-1}}{\de}\rg)^{\chi_n x} \lf(\frac{\de}{\de + \rho_n^{-1}}\rg)^{\chi_n y} \lf(\chi_n \frac{F(0_n, t_n; \de_n)}{F(0_n, s_n; \de_n)} \frac{1}{4 \pi^2}\int_{\scrC_w}\int_{\scrC^-_z} \frac{|F(0_n, s_n; w)|}{|F(0_n, t_n ; z)|}  \frac{\,dz\,dw}{|w(w-z)|} \rg) \\
&\le c\lf(\frac{\de - \rho_n^{-1}}{\de}\rg)^{\chi_n x} \lf(\frac{\de}{\de + \rho_n^{-1}}\rg)^{\chi_n y}.
\end{align*}
	Here in the first inequality, we have brought out the terms depending on $x$ and $y$ and used that the contours $\scrC_w$ and $\scrC^-_z$ live in/out of the disk of radius $\de \pm \rho_n^{-1}$ as in Propositions \ref{P:G1-bounds-w} (III) and \ref{P:G1-bounds-z} (III). For the second inequality, we require that the remaining term inside the brackets is uniformly bounded in $n$. This follows from the bounds carried out in Step 4 in the special case when $\Om = 0$.
%
%
Since $\chi_n = \de \rho_n$ and $\rho_n^{-1} = o(\de)$ as $n \to \infty$, the right hand side above is then bounded by $c e^{-d(x + y)}$, as desired.
\end{proof}

\section{Uniform convergence}\label{S:Uniform}

In this section, we use the Gibbs property to upgrade the finite dimensional distributional convergence of nonintersecting walks to uniform-on-compact convergence. This will finish the proof of Theorem \ref{T:main-walk}. Using the Gibbs property to control the ensemble  also features in the main theorem of \cite{CH}, which applies to nonintersecting Brownian motions.

Given real numbers $s<t$, $x$ and $y$, let $\nu(s,x;t,y)$ denote the law of a Brownian bridge $B : [s,t]\to \mathbb R$ with variance $2$ and $B(s)=x, B(t)=y$. We call $B$ a Brownian bridge from $(s, x)$ to $(t, y)$. The law $\nu(s,x;t,y)$ is a probability measure on the space of continuous functions equipped with the topology of uniform convergence and the Borel $\sigma$-algebra.

For each $n$, we also consider a collection of random walk bridge laws $\eta = \eta_n(s, x; t, y)$. This collection will take values on a countable discrete mesh $(s, x; y, t) \in \mathbb{L}_n \sset \R^4$, such that the following conditions hold:
\begin{itemize}[nosep]
    \item 
$\mathbb{L}_n \cap K \to \{(s, x; t, y) \in K : s \le t\}
$
in the Hausdorff metric as $n \to \infty$ for every compact set $K \sset \R^4$. Here recall that the Hausdorff metric on closed subsets of $\R^n$ is given by
\begin{equation}
\label{E:hausdorff}
d_H(A, B) = \max \{\sup_{x \in A} \inf_{y \in B} \|x - y\|_2, \sup_{y \in B} \inf_{x \in A} \|x - y\|_2\}.
\end{equation}
Note that this may be infinite for certain choices of closed subsets.
\item The projection of $\mathbb{L}_n$ onto the first and third coordinates is equal to $a_n^{-1} \Z$ for some sequence $a_n \to \infty$.
\end{itemize}

When talking about discrete random walks in this section, we always consider their piecewise linear versions as continuous functions. We ask that our collection satisfies some simple properties. 

\begin{description}

\item[Bridge property:] For every $(s, x; t, y) \in \mathbb{L}_n$, the distribution $\eta_n(s,x,t,y)$ is supported on continuous functions $f:[s,t]\to \mathbb R$
with $f(s)=x,f(t)=y$ such that
\begin{itemize}
    \item $f$ is linear on each segment $[a_n^{-1} \ell, a_n^{-1} (\ell+1)]$.
    \item For every $r \le r' \in [s, t] \cap a_n^{-1} \Z$, the quadruple
    $(r, f(r); r', f(r'))$ is in $\mathbb{L}_n$.
\end{itemize}

\item[Bridge Gibbs property:] Let
$(s, x; t, y) \in \mathbb{L}_n$ and let $u, v \in a_n^{1} \Z$ for some $[u, v] \sset [s, t]$.
If $X\sim\eta_n(s,x,t,y)$, then almost surely,
$$
\p\lf(X|_{[u,v]} \in \cdot \; | \; \sig(X|_{[u,v]^c})\rg) = \eta_n(u,X(u), v,X(v)).
$$
This is an almost sure equality of random measures. The right hand side is well defined since $\eta_n(u,X(u), v,X(v)) \in \mathbb{L}_n$ by the bridge property.

\item[Brownian limit:] For any $(s_n, x_n; t_n, y_n) \in \mathbb{L}_n$ converging to $(s, x; t, y)$ as $n\to\infty$ with $s < t$, we have $\eta_n(s_n,x_n,t_n,y_n) \to \nu(s, x; t, y)$. Here the underlying topology is the weak topology on the space of Borel measures on continuous functions from $\R \to \R$ with the uniform-on-compact topology. We can associate $\eta_n, \nu$ to measures on this space by associating any continuous function $f:[s, t] \to \R$ with the continuous function $f':\R \to \R$ which equals $f$ on $[s, t]$, and is constant on $(-\infty, s]$ and $[t, \infty)$.

\end{description}

%

A set $D\subset \mathbb R^2$ is called \textbf{downward closed} when $D$ is closed and has the property that $(t,x')\in D$ implies  $(t,x)\in D$ for all $x<x'$. For any downward closed set, we can define $u_D:\R \to \R\cup \{\pm \infty\}$ by
$$
u_D(t) = \sup \{y \in \R : (y, t) \in D\}. 
$$
Since $D$ is closed, $u_D$ is necessarily upper semicontinuous. Moreover, since $D$ is downward closed we have
\begin{equation}
\label{E:uD}
D = \{(x, y) \in \R^2 : y \le u_D(x)\}.
\end{equation}
Now consider a sequence of $k$ endpoints $e= (e_1, \dots, e_k)$ and a downwards closed set $D$. We call the pair $(e, D)$ a \textbf{$k$-admissible configuration} for $\eta_n$ (or $\nu$) if the endpoints $e_i$ are of the form
$e_i = (a,x_i; b,y_i) \in \mathbb{L}_n$ (or $\R^4$), with $a < b$, $x_1 > x_2 > \ldots > x_k$ and $y_1 > y_2 > \ldots > y_{k}$, and there is a positive probability that $k$ independent bridges with laws $\eta_n(e_i)$ (or $\nu(e_i)$) avoid each other and $D$. More precisely, letting $\scrC_{a, b}^k$ denote the space of all $k$-tuples of continuous functions on $[a, b]$ with the topology of uniform convergence, this says that the open set
\begin{equation}
\label{E:AD-def}
V(D) := \{f \in \scrC_{a, b}^k : f_1(s) > f_2(s) > \dots > f_k(s) > u_D(s), \;\; \forall s \in [a, b]\}
\end{equation}
has positive $\eta_n(e_1) \X \cdots \X \eta_n(e_k)$-measure (or $\nu(e_1) \X \cdots \X \nu(e_k)$-measure). Note that a configuration $(e, D)$ with $e_i = (a,x_i; b,y_i) \in \R^4$ where $a < b$, $x_1 > x_2 > \ldots > x_k$ and $y_1 > y_2 > \ldots > y_{k}$ is $k$-admissible for $\nu$ exactly when $x_k > u_D(a)$, $y_k > u_D(b),$ and $\sup_{s \in [a, b]} u_D(s) < \infty$.

If $(e, D)$ is a $k$-admissible configuration for $\eta_n$, then we let the {\bf nonintersecting bridge law} $\eta^k_n(e,D)$ be the law of $k$ independent bridges with laws $\eta_n(e_i)$ conditioned to avoid each other and $D$. We will also write $\eta_n^k(e) := \eta_n(e_1) \X \cdots \X \eta_n(e_k)$ for the law of $k$ independendent bridges with no ordering condition imposed. We similarly define $\nu^k(e, D), \nu^k(e)$.

\begin{description}
\item[Monotonocity:]
For $a,b$ fixed we say $e\le e'$ if $x_i\le x_i'$ and $y_i\le y_i'$ for all $1\leq i\leq k$. The family $\eta_n$ satisfies the monotonicity property if $\eta_n^k(e,D)$ is stochastically dominated by $\eta_n^k(e',D')$ whenever both $(e, D)$ and $(e, D')$ are $k$-admissible configurations for $\eta_n$, $e\le e'$, and $D\subset D'$. That is, there exists a probability space $(\Om, \mathcal F, \mathbb P)$ on which we can define two random variables $X, Y$ with laws $\eta_n^k(e,D)$ and $\eta_n^k(e',D')$, respectively, such that $\p$-almost surely we have
$
X_i(s) \le Y_i(s)
$
and all $s \in [a, b], i = 1, \dots, k$.
\end{description}

As the conditions above are somewhat abstract, we present the following concrete example of a sequence of rescaled Bernoulli walk measures that satisfy the above conditions. We will later apply the abstract framework of this section to this example to complete the proof of Theorem \ref{T:main-walk}.

\begin{example}
	 \label{E:Bernoulli-walks}
	 First, let 
	 $$
	 \mathbb{L}_0 = \{(s, x; t, y)\in \Z^2: |y - x| \le t-s, \; s + x, y + t \in 2\Z \}.
	 $$
	 For $(s, x; t, y) \in \mathbb{L}_0$, let $\tilde \eta(s, x; t, y)$ denote uniform measure on continuous functions $f:[s, t] \to \R$ such that $f(s) = x, f(t) = y$, and such that on every integer interval $[n, n+1], n \in \{s, \dots, t-1\}$, $f$ is a linear function with slope either $-1$ or $1$. In other words, $\tilde \eta(s, x; t, y)$ is uniform measure on simple random walk paths from $(s, x)$ to $(t, y)$. The family $\tilde \eta$ satisfies the bridge and bridge Gibbs properties; the second is an immediate consequence of the fact that the restriction of uniform measure to a set is still uniform. Moreover, the family $\tilde \eta$ satisfies monotonicity, see the proof of Lemmas 2.6 and 2.7 in \cite{CH}.
	 
	 Families of rescaled Bernoulli random walks that fit into the context of Theorem \ref{T:main-walk} can easily be obtained by rescaling $\tilde \eta$. With notation as in that theorem, let
	 $$
	 B_n = \lf[\begin{array}{cc}
	      \tau_n^{-1} & 0 \\
	      -\chi_n^{-1} \ga' & \chi_n^{-1}
	 \end{array}\rg] \lf[\begin{array}{cc}
	 1 & 0 \\
	1/2  & 1/2
	 \end{array}\rg]
	 $$
	 be a sequence of scaling transformations of the plane. Define associated transformations $\tilde B_n:\R^4 \to \R^4$ and $\hat B_n$ which operate on real-valued functions whose domain is an interval in $\R$ by letting
	 $$
	 \tilde B_n(s, x; t, y) = (B_n(s, x); B_n(t, y)), \qquad \hat B_n f(s) = \chi_n^{-1}\lf(\frac{f(\tau_n^{-1} s) + s}2 - \ga' s\rg).
	 $$
	 Note that $\hat B_n f$ is the function whose graph is given by the graph of $f$, rescaled by $B_n$. Define $\mathbb{L}_n = \tilde B_n \mathbb{L}_0$, and for $u \in \mathbb{L}_n$, define $\eta_n(u)$ by setting
	 $$
	 \eta_n(u)(A) = \tilde \eta(\tilde B_n^{-1} u)(\hat B_n^{-1} A). 
	 $$
	 We can think of $\eta_n$ as the pushforward of the family $\eta$, under the map $B_n$. When thinking about the transformations $B_n$, the leftmost matrix takes simple random walk paths to Bernoulli walk paths, and the second matrix rescales those walks to converge to Brownian motion.
	 
	 Condition (i) in Theorem \ref{T:main-walk} and the relationships between $\tau_n, \chi_n$, guarantee that the collection $\eta_n$ satisfies the Brownian limit property; this is simply weak convergence of simple random walk bridges to a  Brownian bridge as the number of steps $n$ tends to infinity and the displacement of the endpoints of the bridge is $O(\sqrt{n})$. The remaining properties are inherited from $\tilde \eta$.
	 \end{example}

\subsection{Convergence Theory}

For a sequence of functions $f_1,f_2,\ldots$, reals $a<b$ and $k\in \mathbb N$,  let  $\mathcal E_{k,a,b}f$ denote the $k$-tuple of endpoints $((a, f_1(a);b,f_1(b))$, $\ldots$ $(a, f_k(a); b,f_k(b)))$. We will drop $a,b$ from the subscript in $\mathcal E_{k,a,b}$ when their role is clear.

Let  $\bar {\mathbb R} = {\mathbb R}\cup\{ -\infty, \infty\}$ be the two-point compactification of the real line.
For an interval $I$ and a continuous function $f:I\to \bar{\mathbb R}$ let
$$\bar f=\{(x,y)\in I \times \bar{\mathbb R}\ :\ y\le f(x)\}$$
denote the sublevel set of $f$. This is a downward closed set.

\begin{theorem}[Uniform convergence on compact sets]\label{T:uniform}
	Let $J_n$ be a sequence of real intervals with $J_n \to \R$.
For each $n$, let $\scrA^n=(\scrA^n_1,\scrA^n_2,\ldots, \scrA^n_{\al_n})$ be a random sequence of continuous functions from $J_n \to\mathbb R$ such that almost surely, $\scrA^n_1(x) > \scrA^n_2(x) > \dots$ for all $x \in \R$. We also allow each sequence  $\scrA^n$ to consist of only finitely many functions $\al_n$ so long as $\al_n \to \infty$ with $n$.

Assume that the following conditions hold:

(i) The finite dimensional distributions of $\scrA^n$ converge to $\scrA$, the parabolic Airy line ensemble.

(ii) There exists a family of random walk bridge laws $\eta$ with the bridge, bridge Gibbs, Brownian limit and monotonicity properties such that $\scrA_n$ satisfies the following Gibbs property:

With $\mathbb{L}_n, a_n$ as in the definitions above, for any $k \in \N$ and any compact interval $I$ the following holds for all large enough $n$. 
Consider any  real numbers $c < d \in a_n^{-1} \Z \cap I$, and let $\scrF_{k, c, d}$ denote the $\sig$-algebra generated by all values $\scrL_i(x), (i, x) \notin \{1, \dots, k\} \X (c, d)$.
Then almost surely $$
\p\lf( (\scrA^n_1, \dots, \scrA^n_k)|_{[c,d]} \in \cdot \; | \; \scrF_{k, c, d} \rg) = \eta^k_n(\mathcal E_{k,c, d} \scrA^n,\bar\scrA_{k+1}^n).
$$ 
As in the Bridge Gibbs property above, this is an almost surely equality of random measures.

Then $\scrL^n = (\scrA^n_1, \scrL^n_2 \dots, )$ converges in distribution the parabolic Airy line ensemble $\scrA = (\scrL_1, \scrL_2, \dots)$ in the product topology on sequences of functions from $\R \to \R$, where convergence of individual functions is in the uniform-on-compact topology.
\end{theorem}


By the definition of the uniform-on-compact topology, it suffices to show that for each $k$, and each interval $I = [a, b]$ with $a, b \in \Z$ that the process $(\scrA_1^n, \dots, \scrA_k^n)|_I$ is tight with respect to uniform convergence.  
 Without loss of generality, we can assume that $a_n \in \Z$ for all $n$, and so $\Z \sset a_n^{-1} \Z$. Indeed, if not then we can reduce to this case by the following time reparametrization. Let $m_n$ be a sequence of positive integers such that $m_n a_n \to 1$. Then $\scrL^n(\cdot)$ converges uniformly on compacts to the parabolic Airy line ensemble if and only if $\scrL^n(m_n a_n \cdot)$ does. Moreover, $\scrL^n(m_n a_n \cdot)$ satisfies all the same assumptions as $\scrL_n$, with the family of laws $\eta$ replaced by a time-dilated family of laws $\tilde \eta$ defined on $\mathbb{L}_n^* = \{(m_n a_n s, x; m_n a_n t, y) : (s, x; t, y) \in \mathbb{L}_n\}$. All time coordinates of $\mathbb{L}_n^*$ lie in $m_n^{-1} \Z$.

For the rest of the section, all random walk bridge laws will come from the family $\eta$.
We start with an intuitive convergence lemma for the family $\eta$. 
For this next lemma and moving forward in this section, we will work with downward closed sets defined on $I \X \bar \R$ for some compact interval $I$. The set of closed subsets of $I \X \bar \R$ forms a complete separable space when endowed the metric
$$
d(D_1, D_2) = d_H(f(D_1), f(D_2)),
$$
where $f(x, y) = (x, e^y/(1 + e^y))$ and $d_H$ is the Hausdorff metric on closed subsets of $I \X [0, 1]$. In particular, since we are working on a complete separable metric space we may apply Skorokhod coupling. Moving forward, we will refer to this construction simply as the Hausdorff topology on $I \X \bar \R$. Note that limits of downward closed sets are downward closed in this topology.
\begin{lemma}
	\label{L:eta-cvg}
Fix $k \in \N$, and let $(E^n, D^n)$ be a sequence of random $k$-admissible configurations for $\eta_n$ that converge in distribution to $(E, D)$, a random $k$-admissible configuration for $\nu$.
 Here the underlying topology is pointwise convergence for points $E_n$, and the Hausdorff topology on the sets $D_n$ as described above. Further assume that all the $E^n$ are defined on the same deterministic interval $[a, b]$, e.g. $E^n_i = (a, X^n_i; b, Y^n_i)$ for random $X^n_i, Y^n_i$.
 Then (recalling the notation from \eqref{E:AD-def}), we have
\begin{equation}
\label{E:wantwant}
\eta^k_n(E^n)(V(D^n)) \cvgd \nu^k(E)(V(D)) \qquad \mathand \qquad \eta^k_n(E^n, D^n) \cvgd \nu^k(E, D).
\end{equation}
\end{lemma}

To prove Lemma \ref{L:eta-cvg}, we need the following auxiliary result.

\begin{lemma}
\label{L:cty-set}
For an admissible configuration $(e,D)$ the set $V(D)$ is a continuity set for $\nu^k(e)$.
\end{lemma}

\begin{proof}
Set 
\begin{align*}
X_i(f)&=\inf_s \{ f_i(s)-f_{i+1} (s)\}, \qquad i=1,\ldots k-1, \\
X_k(f)&=\inf_s \{ f_k(s)-u_D(s)\}, \\
M(f)&=\min_i X_i(f).
\end{align*}
It is easy to check that $M$ is a continuous function from $\mathcal{C}_{a,b}^k\to \mathbb R$, and so
$V(D)= \{f:M(f)>0\}$.
Similarly, $\{f:M(f)\ge - \epsilon\}$ is a closed set containing $V(D)$, and 
$$
\partial V(D)\subset \{f:M(f)\ge -\epsilon \}\setminus V(D) = \{f:M(f)\in [-\epsilon,0]\}.
$$
If the law of $M(f)$ under $\nu^k(e)$ has no atom at zero, then the probability of the right hand side tends to zero as $\epsilon\to 0$. For this, it suffices to show that none of the $X_i$ have an atom at zero. 

The random variables $X_1, \ldots ,X_{k-1}$ are each distributed as the minima of Brownian bridges, and hence have no atom at $0$. 
To prove that $X_k$ has no atom at $0$, we write $X_k=\min(y,\inf_{n} Y_n)$, where
\begin{align*}
Y_n=\inf_{s\in [a+1/n,b-1/n]} \{f_k(s)-u_D(s)\}, \qquad y=\min(f_k(a)-u_D(a),f_k(b)-u_D(b)).
\end{align*}
Note that since $u_D$ is upper semicontinuous and $f_k$ is continuous, $f_k - u_D$ is lower semicontinuous. Therefore either $\min(y,\inf_{n} Y_n) = y$ or else $\min(y,\inf_{n} Y_n) = Y_n$ for some value of $n$.
Since $y\not=0$ by admissibility, to prove that $X_k$ has no atom at $0$ it suffices to show that $Y_n$ does not have an atom at zero for each $n$. 
Let 
$$
S_n=\frac{f(a+1/n)+f(b-1/n)}2, \qquad D_n=\frac{f(a+1/n)-f(b-1/n)}2.
$$
and let $\ell_n(s)$ be the linear function with $\ell_n(a+1/n)=1, \ell_n(b-1/n)=-1$. If $f$ is a Brownian bridge on $[a, b]$ with any endpoints, then we have the decomposition
$$
f(s)=S_n(f)+D_n(f)\ell_n(s)+ B(s), \qquad s\in [a+1/n,b-1/n],
$$
where $B$ is a Brownian bridge on $[a+1/n,b-1/n]$ which is $0$ at its endpoints and is independent of $S_n(f), D_n(f)$. Moreover, checking the covariance we see that $S_n, D_n$ are independent normal random variables. Therefore finally
$$
Y_n=S_n(f)+\inf_{s\in [a+1/n,b-1/n]} \{D_n(f)\ell_n(s)+ B(s)\}
$$
is a sum of a normal and an independent random variable, and hence $Y_n$ has no atoms.
\end{proof}

\begin{proof}[Proof of Lemma \ref{L:eta-cvg}]
We first prove this when $E^n, D^n, E, D$ are all deterministic. The key step is to show that for any Borel set $A$ we have that
\begin{equation}
\label{E:Ghh}
|\eta^k_n(E^n)(A \cap V(D)) - \eta^k_n(E^n)(A \cap V(D^n))| \to 0
\end{equation}
with $n$. Given \eqref{E:Ghh}, \eqref{E:wantwant} follows from a few applications of the portmanteau theorem. Indeed, since $\eta^k_n(E^n) \to \nu^k(E)$ weakly and $V(D)$ is a continuity set for $\nu^k$ (Lemma \ref{L:cty-set}) the portmanteau theorem gives that
$$
\eta^k_n(E^n)(V(D)) \to \nu^k(E)(V(D))
$$
as $n \to \infty$. The first convergence in \eqref{E:wantwant} then follows since $|\eta^k_n(E^n)(V(D)) - \eta^k_n(E^n)(V(D^n))| \to 0$, the case when $A$ is the whole space in \eqref{E:Ghh}.
Next, $\eta^k_n(E^n, D^n), \nu^k(E, D)$ are Borel measures on $\scrC_{a, b}^k$ defined by
$$
\eta^k_n(E^n, D^n)(A) := \frac{\eta^k_n(E^n)(A \cap V(D^n))}{\eta^k_n(E^n)(V(D^n))}, \qquad \nu^k(E, D)(A) := \frac{\nu^k(E)(A \cap V(D))}{\nu^k(E)(V(D))}
$$
Therefore again by the portmanteau theorem, for the second convergence in \eqref{E:wantwant}, it suffices to prove that for any open set $A \sset \scrC_{a, b}^k$, we have
\begin{equation}
\label{E:eta-kn}
\liminf_{n \to \infty} \frac{\eta^k_n(E^n)(A \cap V(D^n))}{\eta^k_n(E^n)(V(D^n))} \ge \frac{\nu^k(E)(A \cap V(D))}{\nu^k(E)(V(D))}.
\end{equation}
The set $V(D)$ is open in $\scrC_{a, b}^k$, so for any open set $A$, we have
$$
\liminf_{n \to \infty} \eta^k_n(E^n)(A \cap V(D)) \ge  \nu^k(E)(A \cap V(D))
$$ 
by the portmanteau theorem, since $\eta^k_n(E^n) \to \nu^k(E)$ weakly. The convergence \eqref{E:Ghh} implies that the same claim holds with $A \cap V(D^n)$ in place of $A \cap V(D)$ on the left-hand side above. The claim \eqref{E:eta-kn} then follow
from the first convergence in \eqref{E:wantwant}.

We now show \eqref{E:Ghh}.
For any $n_0 \in \N$, any $n \ge n_0$ and any $A$, the left-hand side of \eqref{E:Ghh} is bounded above by
$$
\eta^k_n(E^n)(V(D) \smin V(D^n)) + \eta^k_n(E^n)(V(D^n) \smin V(D))  \le \eta^k_n(E^n)(\close{ S_{n_0}}) + \eta^k_n(E^n)\lf(\close{T_{n_0}}\rg).
$$
where $S_{n_0} = V(D) \smin (\bigcap_{n \ge n_0} V(D^n))$ and $T_{n_0} = (\bigcup_{n \ge n_0} V(D^n)) \smin V(D)$. Therefore again by the portmanteau theorem, the limsup of the left side of \eqref{E:Ghh} is bounded above by
\begin{equation}
\label{E:nuke}
\nu^k(E)(\close{ S_{n_0}}) + \nu^k(E)\lf(\close{T_{n_0}}\rg).
\end{equation}

Now, since $D^n \to D$ in the Hausdorff topology, as $n_0 \to \infty$ both $\bar S_{n_0}$ and $\bar T_{n_0}$ decrease to a subset of the set 
$$
U = \{f \in \scrC^k_{a, b} : f_1(s) \ge \cdots \ge f_k(s) \ge u_D(s) \;\; \forall s \in [a, b], \mathand f_k(s) = u_D(s) \text{ for some } s \in [a, b]\}.
$$
We have that  $U \sset \del V(D)$. Therefore since $V(D)$ is a continuity set for $\nu^k(E)$ (Lemma \ref{L:cty-set}) we have that $\nu^k(E)(U) = 0$. Therefore continuity from above for the measure $\nu^k(E)$ implies that \eqref{E:nuke} converges to $0$ with $n_0$. This implies \eqref{E:Ghh}, as desired.

We move to the case of random $E^n, D^n$. By Skorokhod's representation theorem (which applies here by the discussion prior to Lemma \ref{L:eta-cvg}), we can consider a coupling where $E^n \to E, D^n \to D$ almost surely. Then by the first part of the proof, the convergences in \eqref{E:wantwant} hold almost surely in this coupling. This implies the desired convergence in distribution.
\end{proof}

For this next lemma we use the same notation and assumptions as in the statement of Theorem \ref{T:uniform} (e.g. $I = [a, b]$). This lemma concerns the sublevel set of the second line, $\bar{\scrA^n_2} = \bar{\scrA^n_2}|_I$. This is a random variable taking values in a compact space, namely $I \times \bar{\mathbb R}$ equipped with the Hausdorff distance (see the discussion prior to Lemma \ref{L:eta-cvg}). 

\begin{lemma}[A uniform upper bound for $\bar{\scrA^n_2}$]\label{L:uni0}
 Any subsequential limit $D$ of $\bar{\scrA^n_2}$  is bounded above:
$$\max \{y:(x,y)\in D\}<\infty \qquad {a.s.}$$
\end{lemma}

\begin{proof}
Let $(X_n,Y_n)\in \bar \scrA_2^n$ be such that $Y_n$ is maximal (to ensure measurability of $X_n$, we can specify that $X_n \le x$ for any other $x$ with $(x, Y_n) \in \scrA_2^n$). Fix $m \in \N$. Next, let $B^n$ be a bridge drawn from the random nonintersecting bridge measure $\eta_n(a-1,\scrA_1^n(a-1); b+1,\scrA_1^n(b+1); D_n)$, where $D_n$ is the downward closed set $\{(X_n, y) \in \R^2 : y \le Y_n \wedge m \}$.

Conditionally on $\scrA_2^n$, $\scrA_1^n(a-1)$, and $\scrA_1^n(b+1)$, the distribution of $\scrA_1^n|_{[a-1, b + 1]}$ is a random walk bridge conditioned to avoid all of $\scrA^n_2|_{[a-1, b + 1]}$. In particular, by the monotonicity property of bridges, $\scrA^n_1$ stochastically dominates $B^n$, since $D_n \sset \bar \scrL^n_2$.
Therefore we can couple $B^n$ together with $\scrA^n_1$, $X_n, Y_n, D_n$ so that $\scrL^n_1 \ge B^n$ on $[a-1, b+1]$.  Note that the monotonicity property of bridges applies is defined for a fixed admissible configuration $(e, D)$ whereas here we work with a \textit{random} admissible configuration. We can deal with this without any technical issues since our random admissible configuration can only take on countably many values.

We take a joint subsequential limit in distribution of $\scrA^n_1(a-1)$, $\scrA^n_1(b+1)$, $X_n$, $D_n$ to get $\scrA_1(a-1)$, $\scrA_1(b+1)$, $X$, $D^*$. By Lemma \ref{L:eta-cvg}, along this subsequence $B^n$ converges in distribution to a Brownian bridge $B$ drawn from the random nonintersection measure $\nu^1(a-1, \scrA_1(a-1); b+1, \scrA_1(b+1); D^*)$. To apply this lemma, we have used that $(a-1, \scrA_1(a-1); b+1, \scrA_1(b+1); D^*)$ is a $1$-admissible configuration for $\nu$ since $D^*$ is bounded above and contained in the set $\{X\} \X \R \sset [a, b] \X\R$. In particular, this means that $B^n$ is tight, so by passing to a possible further subsequence we can take a joint limit of the larger collection of random variables
 $\scrA^n_1(a-1)$, $\scrA^n_1(b+1)$, $\scrA^n_1(c), \scrA^n_2(a)$, $X_n, Y_n, D_n, B^n$ to get $\scrA_1(a-1)$, $\scrA_1(b+1)$, $\scrA_1(c)$,$\scrL_2(a)$,$X$,$Y, D^*, B$ where a priori $Y$ may take the value $\infty$ and here $c := (a + b)/2$,

Since $\scrL^n_1 \ge B^n$ on $[a-1, b+1]$ almost surely for every $n$, amost surely $\scrA_1(c) \le B(c)$.
Hence
$$
\E\scrA_1(c)\ge\E B(c)= \E\big[\E[B(c)\mid B(X), \scrA_1(a-1),\scrA_1(b+1)]\big].
$$ 
Conditionally on $B(X), \scrA_1(a-1),\scrA_1(b+1)$, since $D^* \sset \{X\} \X \R$ the bridge $B$ just consists of two unconditioned Brownian bridges from $(a, \scrA_1(a-1))$ to $(X, B(X))$ and from $(X, B(X))$ to $(b+1, \scrA_1(b+1))$ .
The expectation of a Brownian bridge is a linear function connecting its endpoints. Thus the conditional expectation above is given by a convex combination $W B(X) +(1-W)Y'$, where $W\ge 1/2$ and $Y'$ is either $\scrA_1(a-1)$ or $\scrA_1(b-1)$. The maximality of $Y_n$ implies $Y \geq \scrA_2(a)$. Also, $B(X) \ge Y \wedge m$. Therefore
\begin{eqnarray*}
WB(X) +(1-W)Y'&\ge& W(Y \wedge m)^+ - W(Y\wedge m)^- -(1-W)Y'^-\\&\ge& \frac{1}{2}(Y \wedge m)^+ - (\scrA_2(a)\wedge m)^- -Y'^-
\\&\ge&\frac{1}{2}(Y \wedge m)^+ - \scrA_2(a)^-- \scrA_1(a-1)^--\scrA_1(b+1)^-
\end{eqnarray*}
So we get
$$
\E\scrA_1(c)\ge \E[(Y\wedge m)^+/2]-\E[\scrA_2(a)^-]-\E[\scrA_1(a-1)^-]- \E[\scrA_1(b+1)^-].
$$
The terms involving $\scrA_1,\scrA_2$ above have finite expectation since $\scrA_i(x) + x^2, i \in \N$ is an Airy point process for any $x \in \R$. Taking  $m\to\infty$ we conclude $\E Y^+<\infty$, so $Y<\infty$ a.s.
\end{proof}

\begin{lemma}[Separation]\label{L:uni1}
 Let $k\in \mathbb N$. Any joint subsequential distributional limit $(\mathcal E_k\scrA, D), \scrA_k\left(\frac{a+b}{2}\right)$ of  $(\mathcal E_k\scrA^n, \bar{\scrA}^n_{k+1}), \scrA^n_k\left(\frac{a+b}{2}\right)$  satisfies the following a.s.
 \begin{enumerate}[label=(\Roman*)]
     \item $\scrA_1(t)>\scrA_2(t)>\cdots>\scrA_k(t)$ for $t\in\{a,b\}$,
     \item  $\max \{y:(x,y)\in D\}<\infty$,
     \item  $\left(\frac{a+b}{2},\scrA_k\left(\frac{a+b}{2}\right)\right) \notin D$,
     \item  $\mathcal E_k\scrA\cap D=\emptyset$.
 \end{enumerate}
In particular, $(\mathcal E_k\scrA, D)$ is almost surely a $k$-admissible configuration for $\nu$.
\end{lemma}

\begin{proof} (I) This follows from the convergence of fixed-time measures and the fact that the Airy lines do not intersect at a fixed time.

(II) For $k=1$, this is exactly Lemma \ref{L:uni0}, and for larger $k$ it follows from  monotonicity.

(III) For notational ease, we denote the midpoint $c = \frac{a+b}{2}$. Also let  $S = [\frac{2}{3}a +\frac{1}{3}b, \frac{1}{3}a + \frac{2}{3}b]\times \bar{\mathbb R}$ be the middle third of $[a,b]$ times $\bar{\mathbb R}$.

The  monotonicity property implies that given $\bar{\scrA}_{k+1}^n$ and $\mathcal E_k\scrA^n$, the conditional distribution of $\scrA_1^n,\ldots, \scrA_k^n$ given $\mathcal E_k \scrA^n, \bar{\scrA}^n_{k+1}$ stochastically dominates $k$ random walk bridges $B^n$ drawn from the random nonintersecting bridge measure $\eta^k_n(\mathcal E_kA^n, \bar \scrA^n_{k+1}\cap S)$. In other words, we can couple random continuous functions $B^n = (B^n_1, \dots, B^n_k)$ with $\scrA_1^n,\ldots, \scrA_k^n, \mathcal E_k \scrA^n, \bar{\scrA}^n_{k+1}$ so that given $\mathcal E_k \scrA^n, \bar{\scrA}^n_{k+1}$, we have $B^n \sim \eta^k_n(\mathcal E_kA^n, \bar \scrA^n_{k+1}\cap S)$ and almost surely on $[a, b]$ and for all $i \in \{1, \dots, k\}$ we have $B^n_i \le \scrA^n_i$.


We take a joint subsequential distributional limit of $\mathcal E_k \scrA^n$, $\bar{\scrA}^n_{k+1}$ to get the limit $\mathcal E_k \scrA$, $D$. By (I), the endpoints $\mathcal E_k \scrA$ are separated. Moreover, the set $D \cap S$ is separated from the ends of the interval $[a, b]$, and since $D$ is bounded above by (III), $(\mathcal E_k \scrA, D \cap S)$ is a $k$-admissible configuration. Therefore by Lemma \ref{L:eta-cvg}, along this subsequence $B^n$ converges in distribution jointly with $\mathcal E_k \scrA^n, \bar{\scrA}^n_{k+1}$ to $k$ Brownian bridges $B$ drawn from the random measure $\eta^k(\mathcal E_k \scrA, D \cap S)$. In particular $(c,B_k(c))\notin D\cap S$ a.s. 

Finally consider a further joint subsequential distributional limit $\mathcal E_k \scrA$, $D$, $B$, $\scrA_k(c)$  of $\mathcal E_k \scrA^n$, $\bar{\scrA}^n_{k+1}$, $B^n$, $\scrA^n_k(c)$. Since $D$ is downward closed and almost surely $(c,B_k(c))\notin D\cap S$ and $B^n_k(c) \le \scrA^n_k(c)$ for all $n$, we have $(c,\scrA_k(c))\notin D$ almost surely as well.

(IV) It is enough to show that both $(a, \scrA_k(a)) \notin D$ and $(b, \scrA_k(b)) \notin D$ almost surely. This follows from the conclusion (III) applied to the enlarged intervals $[2a-b,b]$, which has $a$ as its midpoint, and $[a,2b-a]$ which has $b$ as its midpoint.

For the final claim about $k$-admissibility, observe that by (I), (II), and (IV) almost surely we can find $\ep > 0$ and continuous functions $f_1, \dots, f_k:[a, b] \to \R$ such that 
\begin{itemize}[nosep]
	\item $f_i(t) = \scrA_i(t)$ all $i \in \{1, \dots, k\}$ and $t \in \{a, b\}$ 
	\item $\inf \{ |f_i(t) - f_{i+1}(t)| : i \in \{1, \dots, k-1\}, t \in [a, b] \} > 2\ep$
\item The graph of $f_k - \ep$ does not intersect $D$.
\end{itemize}
Let $U_{f, \ep}$ be the set of $k$-tuples of continuous functions $g$ satisfying $\|g_i - f_i\|_\infty < \ep$ for all $i$. Then $U_{f, \ep} \sset V(D)$ and $\nu^k(E)(U_{f, \ep}) > 0$ almost surely, yielding $k$-admissibility.
\end{proof}

\begin{proof}[Proof of Theorem \ref{T:uniform}]
	
Let $\mu_n$ be the joint law of $\mathcal E_k \scrA^n, \bar{\scrA}^n_{k+1}, X$, where conditionally on $\mathcal E_k \scrA^n, \bar{\scrA}^n_{k+1}$, we have $X \sim \eta_n^k(\mathcal E_k \scrA^n)$ and $X$ is independent of $\mathcal L^n$ given $\scrE_k \scrL^n$. Since $\scrE_k \scrL^n$ is tight by assumption (i) of the theorem, $X$ is tight in the uniform topology by the Brownian limit property of $\eta_n$. Since $\bar{\scrA}^n_{k+1}$ is a random variable in a compact topology -- the Hausdorff topology on $I \X \bar \R$ where $I = [a, b]$ -- the whole sequence $\mu_n$ is tight.

Let $\mu^*_n$ be the joint law of $\mathcal E_k \scrA^n$, ${\bar \scrA}_{k+1}^n$, $(\scrA_1^n,\ldots \scrA_k^n)|_I$. Then  $\mu^*_n$ is absolutely continuous with respect to $\mu_n$ with Radon-Nikodym derivative
$$
\frac{d\mu^*_n}{d\mu_n}(e,D,x)=\frac{1}{\eta^k_n(e)(V(D))}1(x \in V(D)).
$$
To prove the theorem, we just need to show that $\mu^*_n$ is tight since this implies tightness of its marginal $(\scrA_1^n,\ldots \scrA_k^n)|_I$. Since $\mu_n$ is tight, it suffices to show that any subsequential distributional limit of the sequence $\eta_n^k(\mathcal E_k \scrA^n)(V(\bar \scrL^n_{k+1})), n \in \N$ is positive almost surely. Indeed, this will imply that for any $\ep > 0$ we can find a $\delta > 0$ such that $\eta_n^k(\mathcal E_k \scrA^n)(V(\bar \scrL^n_{k+1})) > \de$ with probability at least $1 - \ep/2$ for all $n$, and hence for any set $A$ we would have
$$
\mu^*_n(A) < \mu_n(A)/\de + \ep/2.
$$
Therefore letting $K_\ep$ be a compact set such that $\mu_n(K_\ep^c) \le \ep\de/2$ for all $n$ we have that $\mu^*_n(K_\ep^c) < \ep$ for all $n$.

Let $(\mathcal E_k\scrA, D)$ be a joint subsequential limit of $(\mathcal E_k\scrA^n, \bar{\scrA}^n_{k+1})$. The pair $(\mathcal E_k\scrA, D)$ is almost surely a $k$-admissible configuration for $\nu$ by Lemma \ref{L:uni1}, so $\nu^k(E_k\scrA)(V(D)) > 0$ almost surely. By Lemma \ref{L:eta-cvg}, $\eta_n^k(\mathcal E_k \scrA^n)(V(\bar \scrL^n_{k+1})) \cvgd \nu^k(\mathcal E_k\scrA)(V(D))$ along this subsequence, giving the desired positivity.
\end{proof}

\subsection{The Bernoulli case}\label{ss:Bernoulli}

We apply Theorem \ref{T:uniform} in conjunction with the results of Section \ref{S:FDD} to prove Theorem \ref{T:main-walk}.

\begin{proof}[Proof of Theorem \ref{T:main-walk}]
Nonintersecting Bernoulli random walks, scaled according to Theorem \ref{T:main-walk}, converge to the parabolic Airy line ensemble in the finite dimensional distribution sense if and only if $\chi_n \to \infty$; this is the content of Section \ref{S:FDD}. 

In order to use Theorem \ref{T:uniform} we consider the piecewise linear continuous versions of the walk, as opposed to the step function versions. 

Assumption (ii) of Theorem \ref{T:uniform} holds with the rescaled family $\eta$ of Bernoulli bridges given in Example \ref{E:Bernoulli-walks}. The required Gibbs property for assumption (ii) follows immediately from the nonintersection condition. 
An application of Theorem \ref{T:uniform} concludes the proof for the piecewise linear versions. Convergence of the step function versions follows. 
\end{proof}

\section{The geometric environment}\label{S:LPP-geometric}

In this section, we relate nonintersecting geometric random walks to last passage percolation defined in terms of independent geometric random variables. We then use this connection to translate Theorem \ref{T:main-walk} to get Theorem \ref{T:main-lpp}.

The connection is a version of the Robinson-Schensted-Knuth (RSK) correspondence, and is described in terms of last passage percolation with several paths. This approach to RSK, called Greene's theorem, avoids  Young diagrams, Young tableaux and insertion procedures, which are not essential for understanding last passage percolation.

\subsection*{Definition of last passage percolation in a discrete lattice}

\begin{definition} \label{D:LPP-discrete}
	Given nonnegative numbers $(W_{i,j}; \ i,j\in \mathbb{N})$ we define the {\bf last passage value in $W$} to a point $(m, n) \in \N \X \N$ by:
	$$
	L_{n, 1}(m):= \max_{\pi} \sum_{ (a,b) \in \pi} W_{a, b},
	$$
	where the maximum is taken over all possible lattice paths $\pi = (\pi_1, \ldots, \pi_\ell) \in (\mathbb{N} \times \mathbb{N})^\ell$ starting at $(1,1)$ and ending at $(m,n)$ which are of minimal length $\ell = m + n - 1$.  More generally, for any $k \in \mathbb N$, define the {\bf last passage value over $k$ vertex-disjoint paths} by
	
	\begin{equation}
	\label{D:LPP-k}
	L_{n, k}(m) :=
	\max_{\pi^{(1)}, \pi^{(2)}, \ldots, \pi^{(k)}}  \sum_{p=1}^k\sum_{ (a,b) \in \pi^{(p)}} W_{a,b},
	\end{equation}
	where the maximum now is taken over all possible $k$-tuples of \emph{disjoint} minimal length lattice paths, where  the $p$-th path $\pi^{(p)}$, $1\leq p \leq k$, starts at $(1,p)$ and ends at $(m,n-k+p)$. In the case that there are no such $k$-tuples of non-overlapping paths (this happens when $k>\min{(m,n)}$), then we take the convention that $L_{n, k}(m) := L_{n, \min{(m,n)}}  (m) = \sum_{a=1}^m\sum_{b=1}^n W_{a,b} $. We will also set $L_{n,0}=0$.

	\begin{figure}
		\begin{center}
			\begin{tikzpicture}[scale=0.5]
			\draw[thick] (0.5,1.5)--(1.5,1.5)--(1.5,3.5)--(4.5,3.5);
			\draw[thick] (0.5,0.5)--(4.5,0.5)--(4.5,2.5);
			\draw (0,3) rectangle (1,4) node[pos=.5] {0};
			\draw (1,3) rectangle (2,4) node[pos=.5] {3};
			\draw (2,3) rectangle (3,4) node[pos=.5] {0};
			\draw (3,3) rectangle (4,4) node[pos=.5] {3};
			\draw (4,3) rectangle (5,4) node[pos=.5] {0};
			
			\draw (0,2) rectangle (1,3) node[pos=.5] {0};
			\draw (1,2) rectangle (2,3) node[pos=.5] {0};
			\draw (2,2) rectangle (3,3) node[pos=.5] {1};
			\draw (3,2) rectangle (4,3) node[pos=.5] {1};
			\draw (4,2) rectangle (5,3) node[pos=.5] {0};
			
			\draw (0,1) rectangle (1,2) node[pos=.5] {0};
			\draw (1,1) rectangle (2,2) node[pos=.5] {1};
			\draw (2,1) rectangle (3,2) node[pos=.5] {0};
			\draw (3,1) rectangle (4,2) node[pos=.5] {0};
			\draw (4,1) rectangle (5,2) node[pos=.5] {2};
			
			\draw (0,0) rectangle (1,1) node[pos=.5] {0};
			\draw (1,0) rectangle (2,1) node[pos=.5] {0};
			\draw (2,0) rectangle (3,1) node[pos=.5] {0};
			\draw (3,0) rectangle (4,1) node[pos=.5] {4};
			\draw (4,0) rectangle (5,1) node[pos=.5] {0};
			\end{tikzpicture}
		\end{center}
		\caption{The paths $L_{n,k}(m)$ in Definition \ref{D:LPP-discrete} for $n=5, m=4, k=2$.  $W$ is indexed by a quadrant with the \emph{bottom left corner} being $(1,1)$.}
		\label{fig:LPP_stab}
	\end{figure}
\end{definition}

\subsection*{Nonintersecting geometric random walks}

\begin{definition} Consider  a function $f$ which may have jump discontinuities. Let the zigzag graph
	$$\operatorname{graph}^z(f)\subset \mathbb{R}^2$$ of $f$ be the graph of $f$ with each jump discontinuity  straddled by a vertical line segment. We extend this definition to functions $F:\mathbb N\to \mathbb R$ by setting graph$^z(F)=$ graph$^z(F(\floor \cdot))$.
	
	A {\bf geometric  random variable with odds $\beta$} takes the value $k$ with probability $\beta(1+\beta)^{-1-k}$ for $k=0,1,\ldots$. Note that the mean is $1/\beta$.
	
	An ensemble $P_n$ of $n$ \textbf{nonintersecting geometric  walks of odds $\beta$} is a collection of independent random walks $P_{n,i}$ having geometric increments of odds $\beta$,  $P_{n,i}(0)=1-i$ and conditioned to have nonintersecting zigzag graphs. Note that in this case, the nointersecting condition is equivalent to requiring that for all $1\leq i\leq n$ and $t \in \mathbb{N}$,
	$$P_{n, i}(t) < P_{n,i-1}(t-1).$$
	Since the nonintersection event has probability zero, the conditioning must be carried out by taking the $m \to \infty$ limit of conditioning on nonintersection up to time $m$, see \cite{konig2002non} for more details.
\end{definition}

\begin{theorem}[\cite{o2003conditioned}]\label{T:geo_to_L}
	Let $(W_{i,j}; \,i,j\in \mathbb N)$ be independent geometric random variables with odds $\beta$.
	
	Fix $n\in \mathbb N$ and let $L_{n, k}(m)$ be the last passage value across $W$ as in Definition \ref{D:LPP-discrete}.
	
	Let $P_{n, i}$ be a collection of $n$ nonintersecting geometric walks of odds $\beta$.
	
	Then we have the equality in distribution, jointly over all $1 \le k \le n$:
	$$
	P_{n, k}(\cdot) + k - 1 \stackrel{d}{=} L_{n, k}(\cdot) - L_{n, k-1}(\cdot).
	$$
\end{theorem}

\begin{proof}[Proof of the version used in this paper]
	The proof goes by applying the RSK bijection to the array $W$.  Precisely, for any $m \in \mathbb{N}$, if we apply the RSK bijection to $\{ W_{i,j} : 1\leq i \leq m, 1\leq j \leq n\}$, then the length, $\lambda_k(m)$, of the $k$-th row of the resulting Young tableaux has the following two properties:
	
	(1) $\lambda_k(m) = L_{n, k}(m) - L_{n, k-1}(m)$.  This is Greene's theorem, see \cite{sagan2013symmetric}.
	
	(2) The laws of $\{ \lambda_k(\cdot) \}_{1\leq k \leq n}$ and $\{ P_{n, k}(\cdot)+k-1 \}_{1 \leq k \leq n}$ are the same. In fact, this law is given by a certain Doob transform of the unconditioned walks; see Corollary 4.8. in \cite{o2003conditioned}.
\end{proof}

Nonintersecting geometric walks are also known as the Meixner ensemble.

\subsection*{Translation between geometric and Bernoulli random walks}

In this section, we map nonintersecting geometric walks to nonintersecting Bernoulli walks so that Theorem \ref{T:main-walk} can be applied to conclude that the top edge of nonintersecting geometric random walks also converges to the parabolic Airy line ensemble.  The connection between two ensembles is a simple shear transformation. In the case of a single independent random walk, this is self-evident from the relationship between geometric and Bernoulli random variables; in the case of nonintersecting walks the result is still intuitive.

\begin{theorem}
	\label{T:geom-walk} Use the setup of Theorem \ref{T:main-lpp}. For each $n$,  consider $n$ nonintersecting geometric  walks $P_{n, i}, i \in \{1, \dots, n\}$ of odds $\beta_n$.
	Then the following statements are equivalent:
	\begin{enumerate}[label=(\roman*)]
		\item
		\begin{equation}
		\label{E:geom-conds}
		n\to\infty, \qquad m_n\to\infty, \qquad \frac{nm_n}{\beta_n}\to\infty.
		\end{equation}
		\item The rescaled top walks near $m_n$ converge to the parabolic Airy line ensemble $\scrA$:
		$$\frac{(P_{n,k}-h_n)(m_n+\lfloor \tau_{n}t\rfloor)}{\chi_n}\Rightarrow \scrA_k(t).
		$$
	\end{enumerate}
\end{theorem}

\begin{remark}\label{r:11proof}
	Theorem \ref{T:geom-walk} implies Theorem \ref{T:main-lpp} via Theorem \ref{T:geo_to_L} since under (i) the offset of $(k-1)/\chi_n$ coming from the distributional equality in Theorem \ref{T:geo_to_L}  converges to $0$.
\end{remark}

The geometric walks $P_{n, k}$ are precisely related to nonintersecting Bernoulli random walks by a flip and a shear. Let $A$ be the linear map given by the matrix
$$
A=\left[\begin{array}{cc}
1 & 1\\
1 & 0
\end{array}\right],
$$
Then
$
A [\operatorname{graph}^z(P_{n, k})]
$
is the graph of a function $X_{n, k}:[-i+1, \infty) \to \R$ with the properties that $X_{n, k}(0) = 0$ and that $X_{n, k}$ is linear on any interval $[\ell, \ell + 1]$ for $\ell \in \{-i + 1, - i + 2, \dots, \}$.

The following lemma, explicitly relating nonintersecting Bernoulli and nonintersecting geometric random walks, follows by equation 4.78 in \cite{konig2002non}; see also \cite{johansson2002non} for this result at a fixed time.

\begin{lemma}
	\label{L:geom-bern}
	$X_{n, 1}, \dots, X_{n, n}$ are $n$ nonintersecting Bernoulli random walks of odds $\beta$.
\end{lemma}

Using Lemma \ref{L:geom-bern}, we can translate Theorem \ref{T:main-walk} to get Theorem \ref{T:geom-walk}.

\begin{proof}[Proof of Theorem \ref{T:geom-walk}]
	
	In the proof we will use convergence of graphs of functions. To facilitate this, consider the following ``local Hausdorff'' topology $\mathcal T$ of closed subsets of $\mathbb R^2$. A sequence $D_n$ converges to $D$ in  $\mathcal T$ if $D_n\cap [-n,n]\times \mathbb R$ converges in the Hausdorff topology to
	$D\cap [-n,n]\times \mathbb R$ for every  $n\in \mathbb N$.
	
	Now consider functions $f,f_n:\mathbb R\to \mathbb R$ with $f$ continuous. Then $f_n\to f$ uniformly on compacts if and only if the graph of $f_n$ converges to the graph of $f$ in $\mathcal T$. This equivalence also holds for zigzag graphs. In other words $f\mapsto \operatorname{graph} f$ and $f\mapsto \operatorname{graph}^z f$ are functionals which are continuous at $f$ that are continuous.
	
	We now consider how the scaling in Theorems \ref{T:geom-walk}
	and \ref{T:main-walk} acts on the level of graphs. It acts by tranformations from the affine group of the form
	$
	y\mapsto Ax+b
	$
	where $A$ is a $2\times 2$ invertible matrix and $b\in \mathbb R^2$. This group can be represented with $3\times 3$ matrices in the block form as
	$$
	\left(\begin{array}{cc}
	A & b \\
	0 & 1
	\end{array}\right)\left(\begin{array}{c}
	x \\
	1
	\end{array}\right)=
	\left(\begin{array}{c}
	y \\
	1
	\end{array}\right).
	$$
	With this notation we turn to the scaling matrices. Let
	$$
	M_n=
	\left(\begin{array}{ccc}
	\tau_n& 0 & m_n \\
	g'\tau_n &  \chi_n & g\\
	0 & 0 & 1
	\end{array}\right)
	,\quad L_n=
	\left(\begin{array}{ccc}
	\bar\tau_n& 0 & \bar m_n \\
	\bar \ga' \bar \tau_n &  -\bar\chi_n & \bar \gamma\\
	0 & 0 & 1
	\end{array}\right),
	\quad
	B=
	\left(\begin{array}{ccc}
	1& 1 & 0 \\
	1  &  0 & 0\\
	0 & 0 & 1
	\end{array}\right)
	$$
	be the matrices associated to the scaling in Theorem \ref{T:geom-walk} and Theorem \ref{T:main-walk} (the latter distingushed by bars), as well as the transformation taking geometric to Bernoulli walks. Here $g=g_{n, \beta_n}(m_n)$ and $\bar\gamma=\bar\gamma_{n,\beta_n}(\bar m_n)$.
	
	Now assume that condition (i) of Theorem \ref{T:geom-walk} holds. With
	\begin{equation}\label{E:barmn}
	\bar m_n=m_n+g(m_n),
	\end{equation}
	it is straightforward to check that condition (i) of Theorem  \ref{T:main-walk} also holds. The two arctic curves are related by \eqref{E:barmn} and the equality $\bar \gamma(\bar m_n)=m_n$.
	
	By Lemma \ref{L:geom-bern}, $L_n^{-1}B$\,graph$^z(P_{n,k})$ are  graphs of the rescaled nonintersecting Bernoulli random walks.
	By Theorem \ref{T:main-walk} and the continuity of $f\mapsto \operatorname{graph} f$, these converge in law with respect to $\mathcal T$ jointly over $k\in \mathbb N$ to the graphs graph$(\scrA_k)$ of the parabolic Airy line ensemble.
	
	It is straightforward to check that the matrix
	$$
	M_n^{-1} B^{-1} L_n =
	\left(\begin{array}{ccc}
	1 & -\bar{\chi}_n \lf(\bar{\tau}_n \ga'\rg)^{-1}&0\\
	0 & 1&0 \\
	0&0&1
	\end{array}\right)
	$$
	converges to the identity matrix by \eqref{E:bernoulli-scaling} and the fact that $\bar \chi_n \to \infty$. Thus $M_n^{-1}$\,graph$^z(P_{n,k})$ also converges in $\mathcal T$ to graph$(\scrA_k)$ jointly in law over $k\in \mathbb N$. Now, $M_n^{-1}$\,graph$^z(P_{n,k})$ are just the zigzag graphs of the rescaled nonintersecting geometric walks, so the continuity of $\operatorname {graph^z} f\mapsto f$ implies (ii).
	
	For the other direction, if the rescaled geometric walks  converge to the parabolic Airy line ensemble, then the distribution of rescaled last passage values to the point $(m_n, n)$ must converge to a Tracy-Widom random variable by  Theorem \ref{T:geo_to_L}. This requires that the side lengths $m_n,n$ of the relevant box that the expected total sum over this box must approach infinity, yielding \eqref{E:geom-conds}.
\end{proof}

\section{Last passage percolation in other environments}\label{S:Corollaries}

\FloatBarrier

In this section, we consider last passage percolation in other settings, most of which are obtained from suitable limits of the geometric one defined in Section \ref{S:LPP-geometric}. By coupling, we can extend the uniform convergence to the Airy line ensemble to these models. We also consider the Sepp\"al\"ainen-Johansson model, a last passage model which is directly related to Bernoulli walks.

\begin{figure}
	\begin{center}
		\begin{tikzpicture}
		\node [block] (Geo) { \begin{tabular}{rl}
			LPP: & geometric \\
			NI: & geometric \\
			\end{tabular}};
		\node [block, below=1cm of Geo]  (Ber) {\begin{tabular}{rl}
			LPP: & Sepp.-Joh. \\
			NI: & Bernoulli \\
			\end{tabular}};
		\node [block, right=1cm of Geo]  (Exp) {\begin{tabular}{rl}
			LPP: & exponential \\
			NI: & exponential \\
			\end{tabular}};
		\node [block, right=1cm of Ber]  (PLines) {\begin{tabular}{rl}
			LPP: & Poisson lines \\
			NI: & Poisson  \\
			\end{tabular}};
		\node [block, right=1cm of PLines]  (PBox) {\begin{tabular}{rl}
			LPP: & Poisson in plane \\
			NI: & Poisson ($\infty$ lines) \\
			\end{tabular}};
		\node [block, right=1cm of Exp]  (Brown) {\begin{tabular}{rl}
			LPP: & Brownian \\
			NI: & Brownian \\
			\end{tabular}};
		\node [block, right=1cm of Brown]  (Airy) {\begin{tabular}{rl}
			LPP: & directed landscape \\
			NI: & Airy line ensemble \\
			\end{tabular}};
		\path[line, <->, dashed] (Geo) -- (Ber);
		\path[line, ->] (Ber) -- (PLines);
		\path[line, ->, double] (Geo) -- (Exp);
		\path[line, ->] (Geo) -- (PLines);
		\path[line, ->] (PLines) -- (PBox);
		\path[line, ->] (Exp) -- (Brown);
		\path[line, ->, double] (PLines) -- (Brown);
		\path[line, ->, double] (PBox) -- (Airy);
		\path[line, ->] (Brown) -- (Airy);
		\end{tikzpicture}
	\end{center}
	\caption{
		Each model is both a type of last passage percolation and a nonintersecting line ensemble. For last passage percolation with geometric variables there are three discrete quantities: the two lattice coordinates and the last passage value. In the directed landscape, all three are continuous.
		Single arrows indicate degeneration of the lattice coordinates and double arrows indicate degeneration of the values from discrete to continuous. }
\end{figure}

\subsection*{Exponential environment}

\begin{definition}
	\label{def:LPP_exp} The \textbf{exponential random environment} is a last passage percolation model with discrete lattice coordinates and $\mathbb{R}$-valued passage times defined as follows. Let $W :\mathbb{N}^2\to \mathbb{R}$ be defined so that $W_{i,j}$ are independent exponential random variables of mean $1$. For each $n,k,m \in \N$ with $k \le n \wedge m$ define the \textbf{exponential valued LPP time} $L_{n,k}(m)$ to be the passage time with $k$ disjoint paths from the bottom left corner to the top right corner of the box $[1,m] \times [1,n]$ as defined in \eqref{D:LPP-k}.
\end{definition}

\begin{corollary}\label{C:exponential}
	Let $L_{n,k}(m)$ be the last passage time in an exponential environment as defined in Definition \ref{def:LPP_exp}. For any $m,n$ define the arctic curve:
	$$ g_n(m) = n+m+2\sqrt{nm}, $$
	which is the deterministic approximation of the last passage value $L_{n,1}(m)$ in this model.
	
	Let $m_n\to \infty$ be a sequence of natural numbers. Denoting by $g, g', g''$ the value of $g_n$ and its derivatives evaluated at $m_n$, we set the space and time scaling of the model:
	\begin{align}\nonumber
	\label{E:tau-xi-exponential}
	\tau_n^3 &= \frac{2  g '^{2}}{g''^2} = \frac{8m_n^2 (\sqrt{m_n} + \sqrt{n})^2}{n} \qquad &\chi_n^3 = \frac{ g'^{4}}{ -2g''} = \frac{(\sqrt{m_n} + \sqrt{n})^4}{\sqrt{m_nn}}.
	\end{align}
	
	Define the linear approximation to the arctic curve around $m_n$:
	$$
	h_n(m) = g + (m - m_n)g'.
	$$
	Then we have the following convergence in law in the uniform-on-compact topology of functions from $\mathbb N\times \mathbb R \to \mathbb R$:
	\begin{equation}
	\label{E:Lnk-cvg}
\frac{(L_{n,k}-L_{n,k-1}-h_n)(m_n+\lfloor \tau_{n}t\rfloor)}{\chi_n}\Rightarrow \scrA_k(t).
	\end{equation}
\end{corollary}


\begin{proof} The proof goes by coupling the exponential random variables to geometric random variables of very large mean. For each $n$, define an environment $W^n:\N^2 \to \R$ by
	$$
	W^n_{i, j} = \floor{n m_nW_{i, j}}.
	$$
	Each $W^n$ is an environment of i.i.d.\ geometric random variables with odds 
	$$
	\be_n = e^{1/(n m_n)}(1- e^{-1/(n m_n)}) = \frac{1 + o(1)}{n m_n}.
	$$
	Let $\tilde L_{n, k}$ denote last passage percolation in the environment $W^n$. Then
	\begin{itemize}
		\item $|n^{-1} m_n^{-1} \tilde L_{n, k}(t) - L_{n, k}(t)| \le 3k$ for any $t \in \{1, \dots, 2n\}, k \in \{1, \dots, n\}$.
		\item We have the following uniform-on-compact convergence in law:
		\begin{equation}
		\label{E:frac-tildeLk}
		\frac{(\tilde L_{n,k}-\tilde L_{n,k-1}-\tilde h_n)(m_n+\lfloor \tilde \tau_{n}t\rfloor)}{\tilde \chi_n}\Rightarrow \scrA_k(t),
		\end{equation}
		where $\tilde h_n, \tilde \tau_n, \tilde \chi_n$ are as in Theorem \ref{T:main-lpp} for the sequence of geometric last passage models defined from $W^n$.
	\end{itemize}
	Here the first claim follows from the fact that in this coupling, $|n^{-1} m_n^{-1} W^n_{i, j} - W_{i, j}| \le n^{-1} m_n^{-1}$ for any $i, j$, and the second claim uses Theorem \ref{T:main-lpp}. Since $n, m_n \to \infty$, the sequence $\be_n \to 0$, and so condition (i) of that theorem holds.
	
	 Now, a calculation shows that the difference of the left sides of \eqref{E:Lnk-cvg} and \eqref{E:frac-tildeLk} converges to $0$ with $n$, which yields the result. This is best seen by explicitly calculating and simplifying the difference (i.e. with computer algebra) and then using the fact that $\be_n = (n m_n)^{-1} (1 + o(1))$ and the first bullet above to recognize that this difference tends to $0$ with $n$ (we provide more details for a very similar computation in the proof of the next corollary).
%
%
%
\end{proof}

\subsection*{Last passage percolation in continuous time}

The following definition is used for last passage percolation in the Poisson lines and Brownian environments.

\begin{definition} \label{D:LPP-cont}
	Let $F_1, F_2, \ldots $ be a collection of cadlag functions from $\R$ to $\R$. Given $n'\le n \in \N$ and $t'\le t \in \R$, a path $\pi$ from $(t',n')$ to $(t,n)$ is a sequence $t' = \pi_{n'-1} \le \pi_{n'} \le \ldots \le \pi_n  = t$. Such paths $\pi$ are naturally interpreted as nondecreasing functions $\pi: [t',t] \to \{n',n'+1,\ldots, n\}$. Define the weight of $\pi$ in $F$ as
	\begin{equation}
	|\pi|_F =  \sum_{i=n'}^n F_{i}(\pi_{i}) - F_i(\pi_{i-1}^-)
	\end{equation}
	where  $F_i(\pi_{i-1}^-)$ denotes the left limit of $F_i$ at $\pi_{i-1}$.
	
	Define the passage value in $F$ by:
	$$L_{n, 1}(t) = \sup_{\pi} |\pi|_F
	$$ where the superemum is over all paths $\pi$ from $(0,1)$ to $(t,n)$.
	Similarly, we define:
	\begin{equation}\label{E:lp-cont}
	L_{n, k}(t) = \sup_{\pi_{1}, \ldots, \pi_{k}} |\pi_1|_F+\ldots +|\pi_k|_F
	\end{equation}
	where the supremum is now over $k$-tuples of disjoint paths $\pi_p$ from $(0,p)$ to $(t,n-k+p)$. Here, disjointness is defined as strict monotonicity between paths as functions of time.
\end{definition}

\subsection*{Poisson lines environment}

\begin{definition}\label{D:Poisson-lines}
	The \textbf{Poisson lines environment} is a last passage percolation model with semi-discrete lattice coordinates and $\mathbb{N}$-valued passage times defined as follows. This is a special case of the passage values $L_{n,k}(t)$ from Definition \ref{D:LPP-cont} where $F_1, F_2, \ldots$ is a collection of $n$ independent Poisson processes, i.e. the increment $F_i(t) - F_i(s)$ are Poisson random variables of mean $t-s$, and non-overlapping increments are independent.
	
\end{definition}

This model, though not frequently studied, goes back to the late 1990s, e.g. see Section 3 of \cite{seppalainen1996hydrodynamic}. It is related to the discrete-space Hammersley process in the same way that Poisson last passage percolation on the plane is related to the continuous-space Hammersley process, see \cite{ferrari2005multiclass}.

\begin{corollary} \label{C:Poisson-lines} Consider the last passage value in the Poisson lines environment as defined in Definition \ref{D:Poisson-lines}.  Let $t_n$ be a positive sequence; we analyze last passage values to points near $(n, t_n)$ across the sequence $F_i$.  Define the Poisson lines arctic curve
	$$
	g_{n}(t) = t+ 2 \sqrt{tn},
	$$
	the deterministic approximation of the last passage value
	$L_{n, 1}(\cdot)$.
	We now define the temporal and spatial scaling parameters $\tau_n$ and $\chi_n$ in terms of the arctic curve $g=g_{n}$ and its derivatives $g',g''$ taken in the variable $t$ at the value $t_n$:
	\begin{equation*}
	\tau_n^3 = \frac{2 g '}{g''^2}=8t^3(1/n+1/\sqrt{nt}), \qquad \chi_n^3 = \frac{ g'^{2}}{ -2g''}=\sqrt{\frac{t}{n}}(\sqrt t+ \sqrt n)^2.
	\end{equation*}
	Also, let $h_n$ be the linear approximation of the arctic curve $g$ at $t_n$:
	$$
	h_n(t) = g + (t - t_n)g'
	$$
	Then the following statements are equivalent:
	\begin{enumerate}[label=(\roman*)]
		\item The number of lines and the mean number of accessible Poisson points converge to $\infty$:
		\begin{equation}\nonumber
		n\to\infty, \qquad nt_n\to\infty.
		\end{equation}
		\item The rescaled differences of the $k$-path and $(k-1)$-path last passage values converge in distribution, uniformly over compact subsets of $\mathbb{N}\times \mathbb{R}$, to the parabolic Airy line ensemble $\scrA$:
		\begin{equation}
		\label{E:LnkLnk}
		\frac{(L_{n,k}-L_{n,k-1}-h_n)(t_n+ \tau_{n}t)}{\chi_n}\Rightarrow \scrA_k(t).
		\end{equation}
	\end{enumerate}
\end{corollary}

\begin{proof} We first show that (i) implies (ii). We convert the Poisson processes into weights on a lattice by counting points in small intervals. Symbolically, given a partition of $\mathbb{R^+}$, $0<s_1<s_2 < \ldots $, we can define $P_{i,j} = F_j(s_{i})-F_j(s_{i-1})$. Because the $F_j$ are independent Poisson processes, each $P_{i,j}$ is an independent Poisson random variable and we can consider last passage values in the grid $P_{i,j}$. As long as the intervals are small enough so that there is at most one Poisson point per column (i.e. For all $i$ in a large range, we have $\sum_{j=1}^n P_{i,j} \leq 1$), the lattice and the Poisson lines last passage values are exactly equal by definition. It will be more convenient to consider last passage values in this discrete grid because this will allow us to use the results we've already developed for discrete lattices.
	
	Given a sequence $t_n$, we will pick $\beta_n$ large enough so that $\be_n t_n \in \N$ for all $n$ and
	\begin{equation}
	\label{E:n2tnn}
	\be_n/\max(n^{100}, t_n^{100}, t_n^{-100}) \to \infty \mathas n \to \infty.
	\end{equation}
	We could get away with something weaker here but it will be easier to deal with computation errors when $\be_n$ is very large. For $j \in \{1, \dots, n\}$ and $i \in \{1, \dots, 2m_n\}$, where $m_n := \beta_nt_n$, define $P_{i,j} = F_i(i/\be_n) - F_i((i-1)/\be_n)$. With this definition, the random variables $P_{i, j}$ form an i.i.d. array of Poisson random variables of mean $1/\beta_n$.
	
	Note that the total variation distance between Poisson and geometric random variables with mean $1/\beta_n$ is at most $10/\beta_n^2$. We can replace each Poisson increment $P_{i,j}$ in our array by a geometric random variable $W_{i,j}$ with the same mean $1/\beta_n$ for a price of $10nm_n/\beta_n^2 = 10 nt_n/\beta_n$ in total variation distance. Therefore in a optimal coupling we can ensure that $\p A^c_n \le 10nt_n/\beta_n$, where
	$$
	A_n = \left\{ W_{i,j} = P_{i,j} \text{ for all } 1\leq i \leq 2m_n, 1\leq j \leq n \right\}.
	$$
	
	Let $B$ be the event that there is a vertical line with index $i\in\{1,\ldots, 2m_n\}$ with total sum $S_i=P_{i,1}+\ldots + P_{i,n}$ more than one, symbolically 
	$$
	B_n= \left\{ \exists i\in\{1,\ldots, 2m_n\} \text{ so that } S_i > 1 \right\}.
	$$
	Since the $S_i$ are  Poisson random variables of mean $n/\be_n$, by a union bound we have that
	$$
	\p B_n \le \sum_{i=1}^{2m_n} \p(S_i > 1) = 2m_n (1 - e^{-n/\be_n}(1 + n/\be_n)) \le 2 \be_n t_n (n/\be_n)^2.
	$$
	Therefore $\p(A_n \smin B_n) \ge 1 - \p A_n^c + \p B_n \ge 1 -2n(5 + n)t_n/\be_n$. This converges to $1$ with $n$ by \eqref{E:n2tnn}.
	
	On the event $A_n\setminus B_n$ the last passage value $L_{n, k}(s)$ coming from the Poisson lines equals the geometric last passage value $\tilde L_{n, k}(\floor{\be_n s})$ in the geometric environment whenever $s \le 2 t_n$.
	Now, the condition \eqref{E:n2tnn} guarantees that the sequence of geometric environments we have defined satisfy the conditions of Theorem \ref{T:main-lpp}. Letting $\tilde \tau_n, \tilde \chi_n, \tilde h_n$ be the corresponding scaling parameters in that theorem, this and the fact that $\p(A_n \smin B_n) \to 1$ with $n$ implies that
\begin{equation}
\label{E:LnkLnkLnk}
	\frac{(L_{n,k}-L_{n,k-1})(t_n+ \tilde \tau_{n}t/\be_n) -\tilde h_n(m_n+ \floor{\tilde \tau_{n}t})}{\tilde \chi_n}\Rightarrow \scrA_k(t),
	\end{equation}
	in the uniform-on-compact topology on $\N \X \R$. 	As in the proof of Corollary \ref{C:exponential}, by comparing the scaling parameters $h_n, \tau_n, \chi_n$ to $\tilde \tau_n, \tilde \chi_n, \tilde h_n$ we can check that this implies \eqref{E:LnkLnk}. Indeed, a computer-aided computation and subsequent simplification shows that
	\begin{equation}
	\label{E:four-big-cpus}
	\begin{split}
	\tilde \tau_n/\be_n &= \tau_n + O((1 + t_n)/\be_n), \\
	\tilde \chi_n &= \chi_n + O((n + t_n)/(\be_n n^{1/4} t_n^{1/4})),\\
	\tilde g_n(m_n) &= g_n(t_n) + O((n +  \sqrt{nt_n})/\be_n),\\
	\be_n \tilde g_n'(m_n) &= g_n'(t_n) + O(\sqrt{n}/(\be_n \sqrt{t_n})).
	\end{split}
	\end{equation}
	Note that all errors on the right hand sides above are very small compared to the main terms by the definitions of $g, \tau_n, \chi_n$ and the condition \eqref{E:n2tnn}. In particular, the error in the relationship between $\tau_n$ and $\tilde \tau_n/\be_n$ implies that $\tau_n \be_n/\tilde \tau_n \to 1$ by \eqref{E:n2tnn} above and the definition of $\tau_n$. Therefore the uniform-on-compact convergence in \eqref{E:LnkLnkLnk} also holds with $\tau_n \be_n$ in place of $\tilde \tau_n$. We can then conclude the convergence in \eqref{E:LnkLnk} by using \eqref{E:four-big-cpus} and \eqref{E:n2tnn} to recognize that the following two differences converge to $0$ with $n$, uniformly over compact sets in $t$:
	\begin{align*}
	\frac{\tilde h_n(m_n+ \floor{\tau_{n} \be_n t})}{\tilde \chi_n} - \frac{h_n(t_n+ \tau_n t)}{\chi_n} &= \frac{\tilde g_n(m_n)}{\tilde \chi_n} - \frac{ g_n(t_n)}{\chi_n} + \frac{\tilde g_n'(m_n) \floor{\tau_{n} \be_n t}}{\tilde \chi_n} - \frac{ g_n'(t_n) \tau_{n} t}{\tilde \chi_n} \\
	\frac{(L_{n, k} - L_{n, k-1})(t_n + \tau_n t)}{\tilde \chi_n}& - \frac{(L_{n, k} - L_{n, k-1})(t_n + \tau_n t)}{\chi_n}.
	\end{align*}
	To see why the second difference converges to $0$, note that we crudely bound the magnitude of $L_{n, k}$ by a trivial by the number of Poisson points in the grid, which is $O(n t_n)$.
	
	To show that (ii) implies (i), first observe that we must have the number of lines tending to infinity in order to define arbitrarily many disjoint paths. Similarly, if the expected number of points does not tend to infinity, then random variables with continuous distributions cannot appear in the limit.
\end{proof}

\subsection*{Brownian last passage percolation}

\begin{definition}\label{D:Brownian}
	The \textbf{Brownian environment} is a last passage percolation model with semi-discrete lattice coordinates and $\mathbb{N}$-valued passage times defined as follows. This is a special case of the passage values $L_{n,k}(t)$ from Definition \ref{D:LPP-cont}, where $F_1 = B_1, F_2 = B_2, \ldots$ is a collection of $n$ independent standard Brownian motions.
\end{definition}
\begin{corollary}\label{C:Brownian}
Define the Brownian arctic curve
	$$
	g_{n}(t) = 2 \sqrt{tn},
	$$
	which is the deterministic approximation of the Brownian last passage value $L_{n, 1}(\cdot)$.
	We now define the temporal and spatial scaling parameters $\tau_n$ and $\chi_n$ in terms of the arctic curve $g=g_{n}$ and its derivatives $g',g''$ taken in the variable $t$ at the value $1$:
	\begin{equation*}
	\tau_n^3 = \frac{2}{g''^2}=\frac{8}{n}, \qquad \chi_n^3 = \frac{1}{ -2g''}=\frac{1}{\sqrt{n}}.
	\end{equation*}
	Also, let $h_n$ be the linear approximation of the arctic curve $g$ at $1$:
	$$
	h_n(t) = g + (t - 1)g'
	$$
	Then the rescaled differences of the $k$-path and $(k-1)$-path last passage values converge in distribution, uniformly over compacts of $\mathbb{N}\times \mathbb{R}$, to the parabolic Airy line ensemble $\scrA$:
	$$\frac{(L_{n,k}-L_{n,k-1}-h_n)(1+ \tau_{n}t)}{\chi_n}\Rightarrow \scrA_k(t).
	$$
\end{corollary}

This uniform convergence result is due \cite{CH}.  In our setting, it follows by coupling Brownian motions to geometric walks in sufficiently long thin boxes. The proof is analogous to the other cases, so we omit it.

\subsection*{Poisson last passage percolation in the plane}

\begin{definition}
	\label{D:Poisson-plane}
Consider a discrete subset of $\Lambda\subset \mathbb R\times [0,1]$ with distinct first and second coordinates. Last passage from $(0,0)$ to $(t,1)$  can be defined by putting a fine enough grid on the box $[0,0]\times[t,1]$ and defining $W_{ij}$ as the number of elements of $\Lambda$ in the box $(i,j)$. It is easy to check that last passage values across the variables $W_{ij}$ stabilize as the mesh of the grid converges to $0$, giving $L_k(t)$. When $\Lambda$ is a Poisson point process, this model is called \textbf{Poisson last passage percolation} in the plane.
\end{definition}

 This corollary covers all planar Poisson convergence results up to a simple affine transformation.

\begin{corollary}[Poisson last passage in the plane]\label{C:Poisson-box}
	Let $s\to \infty$.
	To match with the main theorem, we define the arctic curve
	$$
	g(s) = 2\sqrt{s}
	$$
	which is the deterministic approximation of the Poisson last passage value
	$L_{1}(s)$.
	We now define the temporal and spatial scaling parameters $\tau_s$ and $\chi_s$ in terms of the arctic curve $g$ and its derivatives $g',g''$ taken in the variable $s$:
	\begin{equation*}
	\tau_s^3 = \frac{2g'}{g''^2}=8s^{5/2}, \qquad \chi_s^3 = \frac{g'^2}{ -2g''}=\sqrt{s}.
	\end{equation*}
	Also, let $h_s$ be the linear approximation of the arctic curve $g$ at $s$:
	$$
	h_s(t) = g(s) + (t - s)g'(s)
	$$
	Then the rescaled differences of the $k$-path and $(k-1)$-path last passage values converge in distribution, uniformly over compacts of $\mathbb{N}\times \mathbb{R}$, to the parabolic Airy line ensemble $\scrA$:
	$$\frac{(L_{k}-L_{k-1}-h_s)(s+ \tau_{s}t)}{\chi_s}\Rightarrow \scrA_k(t).
	$$
\end{corollary}

The convergence for the finite dimensional distributions of the top line was shown in  \cite{borodin2006stochastic}, see also \cite{prahofer2002scale} for last passage values along a diagonal line, the polynuclear growth model.

\begin{proof}
	Pick $n=n_s$ so that $s^2/n_s\to 0$ as $s\to\infty$. Consider an $n \times 2n$ grid with vertical spacing $1/n$ and horizontal spacing $s/n$.
	We define the random variables $P_{i,j}$, $1\le i\le n$, $1\le j \le 2n$ by counting points in the corresponding grid boxes. More precisely, for each $i, j$ set $P_{i, j} = \Lambda[s(j-1)/n, sj/n] \X [(i-1)/n, i/n]$, where $\Lambda$ is the original Poisson process.
	
	We couple each Poisson random variable $P_{i,j}$ to a geometric $W_{i,j}$ with the same mean $s/n^2$. As in the proof of Corollary \ref{C:Poisson-lines}, in an optimal coupling  they are all equal with probability at least $1-20n^2(s/n^2)^2$. Finally, we must ensure that each column and row has total sum at most one. Since the number of points in each row and column is Poisson of mean $2s/n$ and $s/n$ respectively, as in the proof of Corollary \ref{C:Poisson-lines}, the probability that there is at most one entry in each row and column is at most $1 - 3n(2s/n)^2$. As in the proof of Corollary \ref{C:Poisson-lines}, on these high probability events Poisson last passage values in the plane equal last passage values on the grid. The claim now follows from Theorem \ref{T:main-lpp}, with details that are analogous to those in the proof of Corollary \ref{C:Poisson-lines}.
\end{proof}

\subsection*{The Sepp\"al\"ainen-Johansson model}

Nonintersecting Bernoulli random walks of Theorem \ref{T:main-walk} are directly related to a different last passage percolation model.
\begin{definition}
	\label{D:SJ-model}
Consider a semi-infinite array $W:\N^2\to \{0,1\}$ where each $W_{i, j}$ is an independent Bernoulli random variable with mean $\be/(1+\be)$, or equivalently, odds $\beta$. Define the last passage value
$$
L_{n, 1}(m) = \sup_\pi \sum_{(i, j) \in \pi} W_{i, j}.
$$
Here the supremum is taken over all paths $\pi = (i, \pi_i)_{i \in \{1, \dots, m\}}$ in the box $\{1, \dots, m\} \X \{1, \dots, n\}$ where $\pi_i$ is a nondecreasing sequence. These are no longer up-right lattice paths, but rather they are forced to have \emph{exactly} one coordinate in each column. This is called the \textbf{Sepp\"al\"ainen-Johansson model}.
This model was defined and the arctic curve was obtained in \cite{seppalainen1998exact}. The fluctuations of $L_{n, 1}(m)$ were analyzed in \cite{johansson2001discrete}. As with usual lattice last passage percolation, we can also define
\begin{equation}\label{E:SJL}
L_{n, k}(m) = \sup_{\pi^1, \dots, \pi^k} \sum_{\ell = 1}^k \sum_{(i, j) \in \pi_\ell} W_{i, j},
\end{equation}
where the paths $\pi^j$ are strictly ordered ($\pi^j_i < \pi^{j-1}_i$ for all $i, j$) and still have exactly one coordinate in each column.
\end{definition}

  For fixed $n$, the functions
$$
L_{n, k}(m) - L_{n, k-1}(m) + n - k
$$
have the law of $n$ nonintersecting Bernoulli walks. This is essentially proven in \cite{o2003conditioned}, Section 4.5. More precisely, combining the results of that section with the results of \cite{konig2002non} shows that the dual RSK algorithm applied to a matrix of independent Bernoulli random variables gives nonintersecting Bernoulli walks. The fact that dual RSK gives differences of last passage values follows from an analogue of Greene's theorem in that context, see \cite{krattenthaler2006growth}.

Our Theorem \ref{T:main-walk} applied to the top walk (rather than the bottom walk) immediately yields the following convergence.

\begin{corollary}
	\label{C:seppalainen}
	Consider sequences of parameters $\beta_n \in (0, \infty)$, $m_n \in \N$ with $\be_n n < m_n$. Let $L_{n,k}$ be the last passage values \eqref{E:SJL} in the Sepp\"al\"ainen-Johansson model.  Define the Nordic curve
	$$
	g_{n, \beta}(m) = m - \frac{(\sqrt{ m} -\sqrt{n \be})^2}{1 + \beta} \indic(m > n \be),
	$$
	the deterministic approximation of the last passage vaue $L_{n, 1}(m).$
	We define scaling parameters $\chi_n$ and $\tau_n$ in terms of $g =g_{n, \be_n}$ and its derivative $g'$, $g''$ evaluated at the point $m_n$:
	\begin{equation*}
	\tau_n^3 = \frac{2g'(1 - g')}{(g'')^2}, \qquad \chi_n^3 = \frac{[g'(1 - g')]^2}{-2g''}.
	\end{equation*}
	Also, let $h_n$ be the linear approximation of $g$ at $m_n$.
	$$
	h_n(m) = g + (m - m_n)g'
	$$
	Then the following are equivalent:
	\begin{enumerate}[label=(\roman*)]
		\item $\chi_n \to \infty$ with $n$.
		\item The rescaled differences between the $k$-path and $(k-1)$-path last passage values converge in distribution, uniformly over compact sets of $\mathbb{N}\times \mathbb{R}$, to the parabolic Airy line ensemble $\scrA$:
		$$
		\frac{(L_{n,k}-L_{n,k-1}-h_n)(m_n + \lfloor\tau_n t\rfloor)}{\chi_n}\Rightarrow \scrA_k(t).
		$$
	\end{enumerate}
	
\end{corollary}

\noindent {\bf Acknowledgments.}  D.D. was supported by an NSERC CGS D scholarship. M.N. was supported by an NSERC postdoctoral fellowship. B.V. was supported by the Canada Research Chair program, the NSERC Discovery Accelerator grant, the MTA Momentum Random Spectra research group, and the ERC consolidator grant 648017 (Abert).

\bibliographystyle{dcu}

\bibliography{UniformToAiryCitations}

\bigskip\bigskip\noindent

\noindent Duncan Dauvergne, Department of Mathematics, University of Toronto, Canada,\\ {\tt duncan.dauvergne@mail.utoronto.ca}

\bigskip

\noindent Mihai Nica, Department of Mathematics, University of Toronto, Canada,\\ {\tt mnica@math.utoronto.ca}

\bigskip

\noindent B\'alint Vir\'ag, Departments of Mathematics and Statistics, University of Toronto, Canada,\\ {\tt balint@math.toronto.edu}

\end{document}